\documentclass[11pt]{article}

\usepackage{latexsym}
\usepackage{amssymb}
\usepackage{amsthm}
\usepackage{amscd}
\usepackage{amsmath}
\usepackage{mathrsfs}
\usepackage{graphicx}
\usepackage{hyperref}
\usepackage{shuffle}
\usepackage{ytableau}
\usepackage{stmaryrd}

\usepackage[colorinlistoftodos]{todonotes}
\usetikzlibrary{chains,scopes,decorations.markings}

\usepackage{mathdots}

\theoremstyle{definition}
\newtheorem* {theorem*}{Theorem}
\newtheorem* {conjecture*}{Conjecture}
\newtheorem{theorem}{Theorem}[section]

\theoremstyle{definition}

\newtheorem* {example*}{Example}

\newtheorem{lemma}[theorem]{Lemma}
\theoremstyle{definition}
\newtheorem{definition}[theorem]{Definition}
\theoremstyle{definition}

\newtheorem{proposition}[theorem]{Proposition}
\newtheorem{corollary}[theorem]{Corollary}

\newtheorem*{remark*}{Remark}
\theoremstyle{definition}
\newtheorem{remark}[theorem]{Remark}
\theoremstyle{definition}
\newtheorem {example}[theorem]{Example}
\theoremstyle{definition}

\theoremstyle{definition}

\theoremstyle{definition}

\theoremstyle{definition}

\newcommand{\ytabc}[2]{
\ytableausetup{boxsize = #1,aligntableaux=center}
{\small\begin{ytableau}  #2  \end{ytableau}}
}

\def\({\left(}
\def\){\right)}

\newcommand{\cP}{\mathcal{P}}

\def\RR{\mathbb{R}}
\def\ZZ{\mathbb{Z}}

\def\spanning{\textnormal{-span}}

\newcommand{\one}{{1\hspace{-.11cm} 1}}

\def\barr{\begin{array}}
\def\earr{\end{array}}
\def\ba{\begin{aligned}}
\def\ea{\end{aligned}}
\def\be{\begin{equation}}
\def\ee{\end{equation}}

\def\qquand{\qquad\text{and}\qquad}
\def\quand{\quad\text{and}\quad}

\def\quord{\quad\text{or}\quad}

\def\ds{\displaystyle}

\def\PP{\mathbb{P}}

\def\ben{\begin{enumerate}}
\def\een{\end{enumerate}}

\def\bei{\begin{itemize}}
\def\eei{\end{itemize}}

\newcommand{\xRightarrow}[2][]{\ext@arrow 0359\Rightarrowfill@{#1}{#2}}

\definecolor{darkred}{rgb}{0.7,0,0} 
\newcommand{\defn}[1]{{\color{darkred}\emph{#1}}} 

 \def\Sym{\Lambda}

 \def\PP{\ZZ_{>0}}

\def\D{\mathsf{D}}
\def\SD{\mathsf{SD}}

\numberwithin{equation}{section}
\allowdisplaybreaks[1]

\def\mSym{\widehat\Lambda}

\def\doublebar{/\hspace{-1mm}/}

\newcommand{\inn}{\mathrm{Int}}

\newcommand{\fkR}{\mathfrak{r}}
\newcommand{\bTheta}{\Theta^{(\beta)}}
\newcommand{\bOmega}{\Omega^{(\beta)}}
\newcommand{\bPhi}{\Phi^{(\beta)}}
\newcommand{\bPsi}{\Psi^{(\beta)}}

\newcommand{\bfx}{\mathbf{x}}
\newcommand{\bfy}{\mathbf{y}}

\def\YY{\mathbb{Y}}

\def\SY{\mathbb{S}\mathbb{Y}}
\renewcommand{\ss}{{\slash\hspace{-1mm}\slash}}

\def\zero{\mathbf{0}}
\def\one{\mathbf{1}}
\def\Sym{\mathsf{Sym}}
\def\mSym{\mathfrak{m}\Sym}

\def\SymSh{\mathsf{SSym}}

\def\bGamma{\Gamma^{(\beta)}}
\def\pGamma{\Gamma}
\def\mGamma{\overline{\Gamma}}

\def\bG{G^{(\beta)}}
\def\pG{G}
\def\mG{\overline{G}}

\def\bGammaP{\Gamma_P^{(\beta)}}
\def\bGammaQ{\Gamma_Q^{(\beta)}}

\def\pGammaP{\Gamma_P}
\def\pGammaQ{\Gamma_Q}

\def\mGammaP{\overline{\Gamma}_P}
\def\mGammaQ{\overline{\Gamma}_Q}

\def\bGP{GP^{(\beta)}}
\def\pGP{GP}
\def\mGP{\overline{GP}}

\def\bGQ{GQ^{(\beta)}}
\def\pGQ{GQ}
\def\mGQ{\overline{GQ}}

\newcommand{\htimes}{\mathbin{\hat\otimes}}

\def\NN{\ZZ_{\geq0}}
\def\corners{\mathsf{Inn}}
\def\free{\mathsf{Free}}
\def\gaps{\mathsf{Gap}}
\def\exts{\mathsf{Adj}}
\def\diff{\mathsf{Diff}}

\def\SVT{\mathsf{SVT}}
\def\ShSVT{\mathsf{ShSVT}}

\newcommand{\wphi}{\widehat\phi}
\newcommand{\wpi}{\widehat\pi}
\newcommand{\wvarepsilon}{\widehat\varepsilon}

\newcommand{\vphi}{\widehat\phi}
\newcommand{\vpi}{\widehat\pi}
\newcommand{\vvarepsilon}{\widehat\varepsilon}

\def\wone{\widehat{1}}

\usepackage{fullpage}

\begin{document}
\title{Positive specializations of $K$-theoretic Schur $P$- and $Q$-functions}
\author{
Eric Marberg \\
    Department of Mathematics \\
    Hong Kong University of Science and Technology \\
    {\tt eric.marberg@gmail.com}
}

\date{}

\maketitle

\begin{abstract}
Yeliussizov has classified the  positive specializations of symmetric Grothendieck functions, defined in several different ways, 
providing a $K$-theoretic lift of the classical Edrei--Thoma theorem.
This note studies the analogous classification problem for 
Ikeda and Naruse's $K$-theoretic Schur $P$- and $Q$-functions, which are the shifted versions of symmetric Grothendieck functions.
Our results extend a shifted variant of the Edrei--Thoma theorem due to Nazarov.
We also discuss an application to the problem of determining the 
extreme harmonic functions on a filtered version of the shifted Young lattice.
\end{abstract}


\section{Introduction}

This note extends Yeliussizov's classification of
the positive specializations of symmetric Grothendieck functions from \cite{Yel20}
to their shifted analogues. 
We also discuss an application to the problem of determining the extreme harmonic functions on a filtered variant of the shifted Young lattice.
The rest of this introduction provides an outline of our main results.

\subsection{Positive specializations}

Let $A$ be an (associative and unital) algebra defined over the field of real numbers $\RR$.
A \defn{specialization} of $A$ 
 is an algebra morphism $\varphi : A\to \RR$.
A specialization of $A$ is \defn{positive} relative to a subset $S\subseteq A$
if $\varphi(s) \geq 0$ for all $s \in S$. In this case we also say that $\varphi$ is \defn{$S$-positive}.

Write 
$\RR\llbracket\bfx\rrbracket = \RR\llbracket x_1,x_2,x_3,\dots\rrbracket$
for the ring of formal power series with real coefficients in a countable sequence of commuting variables.
Recall that a \defn{partition} is a weakly decreasing sequence of integers
$\lambda=(\lambda_1\geq \lambda_2\geq \dots \geq 0)$
with finite sum $|\lambda|$, and that $\mu\subseteq\lambda$ means that $\mu_i \leq \lambda_i$ for all $i$.

Fix a real number $\beta \in \RR$.
The \defn{symmetric Grothendieck functions} $\bG_\lambda$ (indexed by arbitrary partitions $\lambda$  \cite{FK1994,LS1983})
and their skew analogues $\bG_{\lambda\ss\mu}$ (indexed by pairs of partitions $\mu\subseteq \lambda$ \cite{Yel19})
are certain elements of $\RR\llbracket\bfx\rrbracket$ that are symmetric under all permutations
of the $x$-variables. For the precise definition, see Section~\ref{set-tab-sect}.

When $\beta=-1$, these power series are significant in $K$-theory as representatives of the classes of the structure sheaves of Schubert varieties 
in the complex Grassmannian \cite[\S8]{Buch2002}.
When $\beta=0$, we recover the usual (skew) Schur functions
$ s_\lambda = G^{(0)}_\lambda $ and $ s_{\lambda/\mu} = G^{(0)}_{\lambda\ss\mu}.$

Let $\bGamma$ be the $\RR$-linear span of all skew symmetric Grothendieck functions
$\bG_{\lambda\ss\mu}$. The power series $\bG_\lambda=\bG_{\lambda\ss\emptyset}$ provide an $\RR$-basis for 
this vector space, which is actually a subalgebra of  $\RR\llbracket\bfx\rrbracket$ 
with unit $\bG_0 = 1$ and generators $\bG_n = \bG_{(n)}$ for $n \in \PP$
\cite{Buch2002}.
We say that a specialization of $\bGamma$
is  \defn{$\bG$-positive} if it is positive relative to the set of all  functions $\bG_{\lambda\ss\mu}$.

Yeliussizov \cite{Yel19} has classified this set of specializations
in the following way.
Let 
\be
\pG_{\lambda\ss\mu} = G^{(1)}_{\lambda\ss\mu} \quand \mG_{\lambda\ss\mu} = G^{(-1)}_{\lambda\ss\mu}
\ee
and 
define $\pG_{\lambda}$ and $\mG_{\lambda}$ analogously. Then write $\pGamma =\Gamma^{(1)}$ and $\mGamma=\Gamma^{(-1)}$
for the $\RR$-bialgebras generated by these power series.

As will be clarified in Section~\ref{unbounded-sect}, to classify the 
$\bG$-positive specializations of $\bGamma$
it suffices to identify 
the \defn{$\pG$-positive specializations} of $\pGamma$,
the \defn{$\mG$-positive specializations} of $\mGamma$,
and the \defn{Schur positive specializations} of the algebra of bounded degree symmetric functions $\Sym = \Gamma^{(0)}$. 
 The solution to the third classification problem is well-known (see, e.g., \cite[\S1.2]{Yel20}) and will be reviewed in Section~\ref{sym-sect}.

 In the following theorem,
let $a= (a_1\geq a_2 \geq \dots \geq 0)$ and 
$b = (b_1\geq b_2 \geq \dots \geq 0)$ be sequences of nonnegative real numbers and 
let $\gamma \in \RR_{\geq 0} $. Then define 
\be\label{C-as-eq} \textstyle C=  C(a,b,\gamma) = e^\gamma \prod_{n=1}^\infty \frac{1+a_n}{1-b_n}.\ee
 When this infinite product converges, we necessarily have $\sum_{n=1}^\infty (a_n+b_n)<\infty$.

\begin{theorem}[\cite{Yel20}]
\label{y-thm1}
An algebra morphism $\rho : \pGamma \to \RR$   is a $\pG$-positive specialization of $\pGamma$
if and only if 
 for some choice of 
 $a$, $b$, and $\gamma$
satisfying  
$\max(b) < 1 \leq C  =1+\rho(\pG_1)<\infty
 $ one has  
 \be\label{y-thm1-eq} \sum_{n\geq 0} \rho(\pG_n+\pG_{n+1}) z^{n} = 
C e^{\gamma z} \prod_{n=1}^\infty \frac{1+b_n z}{1-a_n z} .\ee

\end{theorem}

This result is equivalent to an explicit formula for 
any $\pG$-positive specialization of $\pGamma$ in terms of the parameters $a$, $b$, $\gamma$, and $C$;
see Section~\ref{unsigned-sect}.
Results in \cite{Yel20} give a similar (but slightly more subtle) classification of the $\mG$-positive specializations of $\mGamma$,
which we discuss in
Section~\ref{signed-sect}.

 \subsection{Shifted variants}

The purpose of this note  is to extend Theorem~\ref{y-thm1} and its signed analogue
to shifted versions of the bialgebra $\bGamma$.
Recall that a partition $\lambda = (\lambda_1>\lambda_2>\dots\geq 0)$ is \defn{strict} if its nonzero parts are all distinct.
Ikeda and Naruse \cite{IkedaNaruse} introduced the \defn{$K$-theoretic Schur $P$- and $Q$-functions} 
\[\bGP_\lambda \in \bGamma
\quand
\bGQ_\lambda\in \bGamma\]
 for all strict partitions $\lambda$.
These have skew analogues $\bGP_{\lambda\ss\mu} \in \bGamma
$ and  $\bGQ_{\lambda\ss\mu} \in \bGamma$ indexed by pairs of strict partitions $\mu\subseteq \lambda$ \cite{LM}.
The precise definitions are reviewed in Section~\ref{set-tab-sect}.

When $\beta=-1$, these power series are significant in $K$-theory as representatives of the classes of the structure sheaves of Schubert varieties 
in the orthogonal and Lagrangian Grassmannians \cite{IkedaNaruse}.
When $\beta=0$, we recover the classical Schur $P$ and $Q$-functions and their skew versions, namely:
\be
P_\lambda = GP^{(0)}_\lambda\text{ and } Q_\lambda = GQ^{(0)}_\lambda
\quad\text{along with}
\quad 
P_{\lambda/\mu} = GP^{(0)}_{\lambda\ss\mu}\text{ and }Q_{\lambda/\mu} = GQ^{(0)}_{\lambda\ss\mu}.\ee

Let $\bGammaP$ and $\bGammaQ$ be the respective $\RR$-vector spaces spanned by all 
$\bGP_{\lambda\ss\mu}$'s
and
$\bGQ_{\lambda\ss\mu}$'s.
 These vector spaces turn out to be sub-algebras with bases
\be
\bGammaQ = \RR\spanning\left\{\bGQ_\lambda: \lambda\text{ strict}\right\}\subsetneq \bGammaP= \RR\spanning\left\{\bGP_\lambda: \lambda\text{ strict}\right\} \subsetneq \bGamma\subsetneq
\RR\llbracket\bfx\rrbracket.
\ee
We say that specializations of $\bGammaP$ and $\bGammaQ$
are  \defn{$\bGP$-positive} 
and \defn{$\bGQ$-positive} 
if they are positive relative to the respective sets of  skew functions
$\bGP_{\lambda\ss\mu}$
and
$\bGQ_{\lambda\ss\mu}$.
To understand these specializations,
it suffices to treat the cases when $\beta \in \{-1,0,1\}$.
The situation when $\beta=0$ is known \cite{Nazarov} and will be reviewed in Section~\ref{nazar-sect}.
For the other cases, let 
\be
\pGP_{\lambda\ss\mu} = GP^{(1)}_{\lambda\ss\mu},
\qquad
 \pGQ_{\lambda\ss\mu} = GQ^{(1)}_{\lambda\ss\mu},
 \qquad
 \mGP_{\lambda\ss\mu} = GP^{(-1)}_{\lambda\ss\mu},
\qquad
 \mGQ_{\lambda\ss\mu} = GQ^{(-1)}_{\lambda\ss\mu},
\ee
and define $\pGP_{\lambda}$, $\mGP_{\lambda}$, $\pGQ_{\lambda}$, and $\mGQ_{\lambda}$ likewise.
Also let 
\be \pGammaP=\Gamma_P^{(1)} \quand \mGammaP=\Gamma_P^{(-1)}
\quad\text{along with}
\quad
\pGammaQ=\Gamma_Q^{(1)}\quand \mGammaQ=\Gamma_Q^{(-1)}.
\ee
We define
specializations of these algebras to be \defn{$\pGP$-}, \defn{$\mGP$-}, \defn{$\pGQ$-}, or \defn{$\mGQ$-positive}
in the obvious way.

We may now present our first new theorem.
Let $a= (a_1\geq a_2 \geq \dots \geq 0)$  be a sequence of nonnegative real numbers and 
suppose $\gamma \in \RR_{\geq0}$.
Define $\overline{x} = \frac{-x}{1+x}$ for any parameter $x$ and let
\be \textstyle D=  D(a,\gamma) = e^\gamma \prod_{n=1}^\infty \frac{1-\overline{a_n}}{1-a_n}.\ee
 When this infinite product converges, we necessarily have $\sum_{n=1}^\infty a_n<\infty$.
 
\begin{theorem} \label{main-thm}
The $\pGP$-positive specializations of $\pGamma_P$
and the $\pGQ$-positive specializations of $\pGamma_Q$
are each obtained by restricting some $\pG$-positive specialization of $\pGamma$.
More specifically, if 
  $\rho : \pGamma \to \RR$ is an algebra morphism
then the following properties are equivalent:
\ben
\item[(a)] $\rho$ restricts to a $\pGP$-positive specialization of $\pGamma_P$.

\item[(b)] $\rho$ restricts to a $\pGQ$-positive specialization of $\pGamma_Q$.
 
 \item[(c)] For some choice of $a$ and $\gamma$ satisfying $1 \leq D =  1+ \rho(\pGP_1)  < \infty$ one has
 \[
 \sum_{n\geq 0} \rho(\pGQ_n+\pGQ_{n+1}) z^{n} 
 = D^2 e^{2\gamma z} \prod_{n=1}^\infty \frac{1-\overline{a_n} z}{1-a_n z}.\]
\een
\end{theorem}

The equivalence of (a) and (b) is surprising, since negative coefficients are required to express an arbitrary $\pGQ$-function as a linear combination of $\pGP$-functions \cite{ChiuMarberg}.

As with Theorem~\ref{y-thm1}, this result is equivalent to an explicit formula for 
any $\pGP$- or $\pGQ$-positive specialization of $\pGammaP$ or $\pGammaQ$;
see Section~\ref{unsigned-sect}.
We also derive a similar classification of the $\mGP$ and $\mGQ$-positive specializations of $\mGammaP$ and $\mGammaQ$ in 
Section~\ref{signed-sect}.

\subsection{Applications to harmonic functions}

The $\beta=1$ classification results presented
in this introduction
 are more natural algebraically than their $\beta=-1$ versions, since if $\beta\geq 0$ then
 we have 
\be
\bG_{\lambda\ss\mu} \in \RR\spanning\left\{ \bG_\nu\right\},
\quad
\bGP_{\lambda\ss\mu} \in \RR\spanning\left\{ \bGP_\nu\right\},
\quad
\bGQ_{\lambda\ss\mu} \in \RR\spanning\left\{ \bGQ_\nu\right\}
\ee
by results in \cite{Buch2002,Mar2025}; see the discussion in Section~\ref{unbounded-sect}.
Hence, a specialization is $\pG$-, $\pGP$-, or $\pGQ$-positive if and only if it is positive relative to the defining 
basis of $\pGamma$,  $\pGammaP$, or  $\pGammaQ$.

Additionally, Yeliussizov has shown  \cite[Thm.~5.2]{Yel20} that the subset of $\pG$-positive specializations $\rho$ of $\pGamma$
that are \defn{normalized} in the sense that $\rho(\pG_{(1)})=1$
are naturally in bijection
with the extreme points of the convex set of \defn{harmonic functions} on a certain \defn{filtered Young graph}.
In Section~\ref{harmonic-sect}, we prove an extension of this result that relates 
normalized $\pGP$-positive specializations of $\bGammaP$
to
the extreme points of the convex set of harmonic functions on a shifted filtered Young graph.


\subsection*{Acknowledgments}

The author thanks Alimzhan Amanov, Joel Lewis, Damir Yeliussizov for helpful discussions.

This article is based on work supported by the National Science Foundation under grant DMS-1929284 while the author was in residence at the Institute for Computational and Experimental Research in Mathematics in Providence, RI, during the Categorification and Computation in Algebraic Combinatorics semester program.

The author was also partially supported by Hong Kong RGC grants 16304122 and 16304625.

\section{Preliminaries}

This section contains some general background material on specializations of bialgebras
and classical algebras of symmetric functions.

\subsection{Bialgebra specializations}

We assume some familiarity with the definition of a bialgebra and its basic properties.
A good reference for this material is \cite{GrinbergReiner}. 
Suppose  $B$ is a bialgebra defined over $\RR$ with coproduct $\Delta_B$.
Let $\otimes $ be the tensor product over $\RR$ and 
write $\nabla_\RR : \RR \otimes \RR\to \RR$ for the usual multiplication map.

\begin{definition}
The \defn{union} of two specializations $\rho_1,\rho_2 : B \to \RR$
is  the specialization
\be\rho_1\sqcup \rho_2 = \nabla_\RR\circ (\rho_1\otimes \rho_2)\circ \Delta_B.\ee
\end{definition}

As $\Delta_B$ is associative, it holds that $\rho_1 \sqcup(\rho_2\sqcup \rho_3) = (\rho_1\sqcup \rho_2)\sqcup \rho_3$
so we many consider iterated unions and omit all parentheses in expressions for these.
In general the union operation on specializations may not be commutative, 
but if $B$ is cocommutative then $\rho_1\sqcup\rho_2 = \rho_2\sqcup\rho_1$.

\begin{definition}
A subset $S\subseteq B$ is \defn{comultiplicative} if
\[
\Delta_B(s) \in \RR_{\geq 0}\spanning\{ s_1\otimes s_2 : s_1,s_2\in S\}\quad\text{for all $s \in S$.}
\]
\end{definition}

If $S$ has this the property then
any union of $S$-positive specializations is also $S$-positive.
Hence, if $S$ is comultiplicative then
the set of $S$-positive specializations of $B$ form a semigroup.

\subsection{Symmetric functions}\label{sym-sect}

Let $\Sym=\Sym(\bfx)\subset \RR\llbracket\bfx\rrbracket $ denote the ring bounded degree \defn{symmetric functions} \cite[\S I]{Macdonald}.
We may express $\Sym$ as a polynomial ring in two ways as
$ \Sym = \RR[h_1,h_2,h_3,\dots] = \RR[e_1,e_2,e_3,\dots]$
where $h_n$ and $e_n$ are the \defn{complete homogeneous} and \defn{elementary} symmetric functions
\be
\textstyle
h_n = \sum_{1 \leq i_1 \leq i_2 \leq \dots \leq i_n} x_{i_1}x_{i_2}\cdots x_{i_n}
\quand
e_n = \sum_{1 \leq i_1 < i_2 < \dots < i_n} x_{i_1}x_{i_2}\cdots x_{i_n}.
\ee
The $\RR$-algebra $\Sym$ has a bialgebra structure \cite[\S2]{GrinbergReiner} in which the counit $\varepsilon : \Sym \to \RR$ is the   map setting $x_1=x_2=x_3=\dots=0$ and where the coproduct $\Delta : \Sym \otimes \Sym \to \Sym$ satisfies 
\be\label{coproduct-def-eq}
\textstyle\Delta(h_n) = \sum_{i=0}^n h_i \otimes h_{n-i} \quand  \Delta(e_n) = \sum_{i=0}^n e_i \otimes e_{n-i}
\quad\text{for all $n \in \NN$}.
\ee
In this formula, we set $e_0 = h_0=1$.
The coproduct for $\Sym$ may be computed by replacing the variables $x_1,x_2,\dots$ with a doubled sequence 
$x_1,x_2,\dots,y_1,y_2,\dots$ and then applying the natural isomorphism $\Sym(\bfx,\bfy) \xrightarrow{\sim} \Sym\otimes\Sym$.

 There is a unique bialgebra involution 
$\omega : \Sym \to\Sym$ with 
\be\label{omega-def}
\omega(h_n) =e_n\quand \omega(e_n) =h_n\quad\text{for all $n \in \NN$.}
\ee
The cocommutative bialgebra $\Sym$ is graded and connected, and therefore is a Hopf algebra.
Its antipode is the composition of $\omega$ with the evaluation map 
$x_i\mapsto -x_i$ negating all variables \cite[\S2]{GrinbergReiner}.

The Hopf algebra $\Sym$ has a distinguished $\RR$-basis of \defn{Schur functions}
$s_\lambda$ indexed by all partitions, which have a skew generalization $s_{\lambda/\mu}$ indexed by pairs of
partitions $\mu\subseteq\lambda$. 
A succinct combinatorial 
definition is provided by the formulas
$s_\lambda = s_{\lambda/\emptyset}$ and $s_{\lambda/\mu} =\sum_T x^T$,
where the sum is over all \defn{semistandard tableaux} of shape $\lambda/\mu$  \cite[\S I.5]{Macdonald}.

We refer to the specializations of $\Sym$ that are positive with respect to
the set of all skew Schur functions $s_{\lambda/\mu}$ 
as \defn{Schur positive}. 
This is the same as the set of specializations that are positive with respect to 
the basis of Schur functions, since 
each skew Schur function is a $\NN$-linear combination of ordinary Schur functions
\cite[\S I.5, (5.3)]{Macdonald}.

The set of skew Schur functions is comultiplicative  
since $\Delta(s_{\lambda/\mu}) = \sum_{\mu\subseteq\kappa\subseteq\lambda} s_{\kappa/\mu}\otimes s_{\lambda/\kappa}$
\cite[\S I.5, (5.10)]{Macdonald}
 and so the set of Schur positive 
specializations is a commutative semigroup under the union operation.
These specializations have a well-known classification, which we briefly review.

\begin{definition}
For an infinite sequence of nonnegative real numbers $a=(a_1\geq a_2\geq a_3 \geq \dots \geq 0)$ with finite sum,
let $\phi_a : \Sym \to\RR$ and $\varepsilon_a : \Sym \to\RR$ be the maps with the formulas
\[
\phi_a(f) = f(a_1,a_2,a_3,\dots)
\quand
\varepsilon_a = \phi_a \circ \omega.
\]
\end{definition}

When $a$ is a single real number, let $\phi_a =\phi_{(a,0,0,0,\dots)}$ and when $a=(a_1\geq a_2\geq \dots\geq a_k)$ is a finite sequence of real numbers,
let $\phi_a = \phi_{(a_1,a_2,\dots,a_k,0,0,0,\dots)}$. Extend the definition of $\varepsilon_a$ similarly.

\begin{definition}
For any real number $\gamma \geq 0$ write  $\pi_\gamma : \Sym \to\RR$ for the unique algebra morphism
satisfying 
$\pi_\gamma(h_n)= \tfrac{\gamma^n}{n!} $ for all $n \in \NN$.
\end{definition}

By \cite[Lem.~6.5]{Yel20} it holds for any $f \in \Sym$ that
\be\label{limit-eq}
\ds\pi_\gamma(f) = \lim_{N\to \infty} f(\underbrace{\gamma/N,\gamma/N,\dots,\gamma/N}_{N\text{ terms}},0,0,0,\dots).
\ee
Using this, one checks that we also have $ \pi_\gamma(e_n )= \tfrac{\gamma^n}{n!}$. Hence 
\be\label{pi-omega-eq}
\pi_\gamma \circ \omega = \pi_\gamma.
\ee
The following result is equivalent to the classical Edrei--Thoma theorem \cite{Edrei,Thoma}.

\begin{theorem}[{Edei--Thoma; see \cite[Thm.~2.4]{Yel20}}]
\label{schur-thm}
The Schur positive specializations of $\Sym$ are the maps  
$\rho = \phi_\alpha \sqcup \varepsilon_b \sqcup \pi_c$
where  $a=(a_1\geq a_2 \geq \dots \geq 0)$ and $b=(b_1\geq b_2\geq  \dots \geq 0)$
have finite sum and  $\gamma\in \RR_{\geq 0}$. The representation of $\rho$
is unique and it holds that $\textstyle \rho (h_1)= \sum_{n=1}^\infty (a_n+b_n)+\gamma_n.$
\end{theorem}

Let $z$ be a formal variable. Then write
 $H(z) ,E(z) \in  \Sym\llbracket z\rrbracket$ for the power series
\be
H(z) := \sum_{n\geq 0} h_n z^n = \prod_{n\geq 1} \frac{1}{1-x_nz}
\quand
E(z) :=\sum_{n\geq 0} e_n z^n = \prod_{n\geq 1} (1+x_nz) = H(-z)^{-1}.
\ee
Any specialization $\rho : \Sym \to \RR$ extends to an algebra morphism $\Sym\llbracket z\rrbracket \to \RR\llbracket z \rrbracket$ 
with the formula
$\textstyle \rho\( \sum_{n \geq 0} f_n z^n \) =  \sum_{n \geq 0} \rho(f_n) z^n.
$
Under this convention, since $\omega(H(z)) = E(z)$, one has 
\be\label{12cup-eq0}
\textstyle
\phi_a(H(z)) = \prod_{n\geq 1}\frac{1}{1-a_nz},
\quad
\varepsilon_a(H(z)) = \prod_{n\geq 1} (1+a_nz),
\quand
\pi_\gamma(H(z)) = e^{\gamma z}
\ee
as well as
\be\label{12cup-eq00}
\textstyle
\phi_a(E(z)) =\prod_{n\geq 1} (1+a_nz) ,
\quad
\varepsilon_a(E(z)) = \prod_{n\geq 1}\frac{1}{1-a_nz},
\quand
\pi_\gamma(E(z)) = e^{\gamma z}.
\ee
The coproduct formula \eqref{coproduct-def-eq} implies for any algebra morphisms $\rho_i : \Sym \to \RR$ that 
\be\label{12cup-eq}
(\rho_1\sqcup \rho_2)(H(z)) = \rho_1(H(z))\rho_2(H(z))
\quand
(\rho_1\sqcup \rho_2)(E(z)) = \rho_1(E(z))\rho_2(E(z))
\ee
so we have $\rho =\phi_\alpha \sqcup \varepsilon_b \sqcup \pi_\gamma$ 
if and only if 
$
\rho(H(z)) = e^{\gamma z}\prod_{n\geq 1} \frac{1+b_n z}{1-a_nz}.
$
Moreover, we see that 
\be
\phi_a = \phi_{a_1}\sqcup \phi_{a_2}\sqcup \phi_{a_3}\sqcup\cdots 
\quand
\varepsilon_a = \varepsilon_{a_1}\sqcup \varepsilon_{a_2}\sqcup \varepsilon_{a_3}\sqcup\cdots .
\ee

\subsection{Shifted symmetric functions}\label{nazar-sect}

Let $\SymSh$ be the subspace of power series $f \in \Sym$ satisfying the \defn{supersymmetry} property
\[f(-z,z,x_1,x_2,x_3,\dots) \in \RR\llbracket x_1,x_2,x_3,\dots\rrbracket.\]
This subspace is a Hopf subalgebra of $\Sym$, and has a basis given by the \defn{Schur $P$-functions} $P_\lambda$ indexed by all strict partitions $\lambda$; see \cite[\S III.8]{Macdonald}.
Since 
 the scalar field is $\RR$, 
 the \defn{Schur $Q$-functions} $Q_\lambda=2^{\ell(\lambda)} P_\lambda$ form a second basis.
 The power series $P_\lambda$ and $Q_\lambda$ are obtained from the symmetric functions $\bGP_\lambda$ and $\bGQ_\lambda$
 defined in Section~\ref{set-tab-sect} by setting $\beta=0$. 

Any specialization of $\SymSh$ that is positive with respect to the basis $\{P_\lambda\}$
is clearly also positive with respect to the basis $\{Q_\lambda\}$.
The classification of such \defn{Schur $P$-positive} specializations is also known \cite{Nazarov}.
%
%
%
For $n \in \NN$ define 
$ q_n = \sum_{i+j=n} e_i h_j $ 
so that 
\[\textstyle q_0=1,\quad q_1 =2h_1,
\quand
Q(z) := \sum_{n \in \NN} q_nz^n = E(z)H(z) =\prod_{n\geq 1}\frac{1+x_n z}{1-x_n z}.
\] 
 The elements $\{q_n\}$ generate $\SymSh$ as an $\RR$-algebra and are algebraically independent \cite[\S III.8]{Macdonald}.

\begin{theorem}[Nazarov \cite{Nazarov}]
\label{symsh-thm}
A specialization $\rho : \SymSh \to \RR$ is Schur $P$-positive if and only if 
there are real numbers 
$a=(a_1 \geq a_2 \geq \dots\geq0) $
and $\gamma\geq 0$ with $\sum_{n =1}^\infty \alpha_n  < \infty$
such that
\[\sum_{n\geq 0} \rho(q_n) z^n = e^{2\gamma z} \prod_{n\geq 1}\frac{1+a_n z}{1-a_n z}.\]
In this case 
$\textstyle
\rho(h_1) =\tfrac{1}{2}\rho(q_1)= \sum_{n =1}^\infty a_n + \gamma
$ and 
$\rho$ coincides with $  \phi_{a} \sqcup \pi_\gamma$  restricted to $\SymSh.$
\end{theorem}

One can derive this theorem in a direct algebraic way from Theorem~\ref{schur-thm}.
The argument is essentially the same as the proof of Theorem~\ref{main-thm2}, just changing various
constructions and properties involving $\beta=1$ to their simplified versions with $\beta=0$.

\section{Positive K-theoretic specializations}
 
This section contains our main results and applications.
The first two subsections review the precise definitions of $\bG_{\lambda\ss\mu}$,
$\bGP_{\lambda\ss\mu}$, and $\bGQ_{\lambda\ss\mu}$, and the algebraic structures they generate.
Our main new results are Theorems~\ref{main-thm2} and \ref{main-thm3} and Corollary~\ref{main-new-cor}.

\subsection{Set-valued tableaux}\label{set-tab-sect}

Let $D\subseteq E$ be finite subsets of $ \PP\times\PP$ and define the \defn{interior} of $D$ to be the set
\be
\inn(D) = \{ (i,j) \in D: (i+1,j)\in D\text{ or }(i,j+1)\in D\}.
\ee
 For any finite sets $A,B\subset \RR $,
write $A \preceq B$ if either set is empty or if $\max(A) \leq \min(B)$. 

\begin{definition}
A \defn{set-valued tableau} of shape $E\doublebar D$ is a 
map 
$T$ that assigns a finite set $T_{ij} \subset \RR_{>0}$ to each 
$(i,j) \in E\setminus \inn(D)$ such that 
$T_{ij}$ is nonempty if $(i,j) \in E\setminus D$  and such that
\[T_{ij} \preceq T_{i,j+1}\quand T_{ij} \preceq T_{i+1,j}\quad\text{for all $(i,j) \in \PP\times \PP$}\]
under that convention that $T_{ij}=\varnothing$ if $(i,j) \notin E\setminus\inn(D)$.
For any such tableau $T$ define
\[\textstyle |T| = \sum_{(i,j)\in E\setminus \inn(D)} |T_{ij}|
\quand
 x^T =\prod_{(i,j)\in E\setminus \inn(D)} \prod_{a \in T_{ij}} x_{ \lceil a\rceil}.
 \]
 \end{definition}

\begin{example}
Suppose $E = D \sqcup \{ (1,8),(1,9), (2,6), (3,5),(3,6), (4,4)\}$ where
\[D = \{ (1,1),(1,2),(1,3),(1,4),(1,5),(1,6),(1,7), (2,2),(2,3),(2,4), (2,5), (3,3),(3,4)\}.\]
In the following picture, the positions indicated by ``$\ytabc{0.4cm}{\none[\cdot]}$'' represent the elements of $\inn(D)$, 
while the boxes $\ytabc{0.4cm}{*(lightgray)}$ 
and $\ytabc{0.4cm}{\ }$
make up $D\setminus\inn(D)$ and $E\setminus D$, respectively:
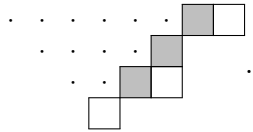
\[
\ytabc{0.4cm}{ \none[\cdot] & \none[\cdot]  & \none[\cdot]  & \none[\cdot]  & \none[\cdot]  & \none[\cdot]  & *(lightgray) & \ \\
\none & \none[\cdot]  & \none[\cdot]  & \none[\cdot]  & \none[\cdot]  & *(lightgray) \\
\none& \none & \none[\cdot]  & \none[\cdot]  & *(lightgray) & \ \\
\none & \none & \none & \
}.
\]
If we set $i' = i-\frac{1}{2}$ then two examples of set-valued tableaux $T$ of shape $E\ss D$ are 
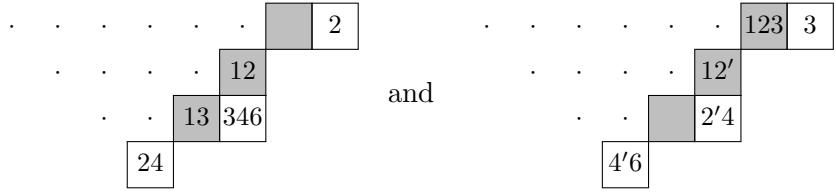
\[
\ytabc{0.6cm}{ \none[\cdot]  & \none[\cdot]  & \none[\cdot]  & \none[\cdot]  & \none[\cdot]  & \none[\cdot]  & *(lightgray)  & 2 \\
\none & \none[\cdot]  & \none[\cdot]  & \none[\cdot]  & \none[\cdot]  & *(lightgray) 12 \\
\none& \none & \none[\cdot]  & \none[\cdot]  &*(lightgray) 13 & 346 \\
\none & \none & \none & 24
}
\quand
\ytabc{0.6cm}{ \none[\cdot]  & \none[\cdot]  & \none[\cdot]  & \none[\cdot]  & \none[\cdot]  & \none[\cdot]  & *(lightgray) 123 & 3 \\
\none & \none[\cdot]  & \none[\cdot]  & \none[\cdot]  & \none[\cdot]  & *(lightgray) 12' \\
\none& \none & \none[\cdot]  & \none[\cdot]  &*(lightgray)  & 2'4 \\
\none & \none & \none & 4'6
}
\]
and in both cases we have $|T|=10$ and $x^T = x_1^2 x_2^3 x_3^2 x_4^2 x_6$.
\end{example}

Fix a parameter $\beta\in \RR$.
Choose partitions $\mu\subseteq\lambda$ and 
 recall that the \defn{diagram} of   $\lambda$ is the set
 \[\D_\lambda = \{ (i, j) : i\in \PP\text{ and }1\leq \lambda_j \leq i\}.\]
  Let $\SVT(\lambda\ss\mu)$  be the collection of set-valued tableaux $T$ of shape $\D_\lambda\ss\D_\mu$
 whose values $T_{ij}$ are all sets of positive integers.
 
 
 \begin{definition}[\cite{Buch2002}]
The \defn{symmetric Grothendieck function} of shape $\lambda\ss\mu$  is
\[\textstyle \bG_{\lambda \doublebar \mu} = \sum_{T\in \SVT(\lambda\ss\mu)} \beta^{|T|-|\lambda/\mu|} x^T \in \RR\llbracket \bfx\rrbracket.
\]
We also define $\bG_{\lambda}=\bG_{\lambda \doublebar \emptyset}$
and for convenience let
$\bG_{\lambda\ss\mu}=0$ when $\mu\not\subseteq \lambda$. 
\end{definition}

\begin{example}
If $0< m\neq n$ then   $\bG_{(n) \doublebar (m)} = \bG_{(n-m)} + \beta \bG_{(n-m+1)}$
where we set $\bG_{(0)}=1$.
\end{example}

Now suppose $\mu\subseteq\lambda$ are strict partitions and
 recall that the \defn{shifted diagram} of   $\lambda$ is the set
 \[\SD_\lambda = \{ (i, i+j-1) : (i,j) \in \D_\lambda\}.\]
Let $\ShSVT_Q(\lambda\ss\mu)$ be the set of set-valued tableaux $T$ of shape $\SD_\lambda\ss\SD_\mu$
with
\be
T_{ij} \subset \tfrac{1}{2}\PP
\quand
T_{ij} \cap T_{i,j+1} \subset\PP
 \quand T_{ij} \cap T_{i+1,j} \subset\PP-\tfrac{1}{2}
\ee for all $(i,j) \in \PP\times \PP$. Also define
\be 
\ShSVT_P(\lambda\ss\mu) = \left\{T \in \ShSVT_Q(\lambda\ss\mu):
T_{ij} \subset \PP\text{ whenever }i=j\right\}.
\ee


\begin{definition}[\cite{IkedaNaruse,LM}]
The \defn{$K$-theoretic Schur $Q$-function} of shape $\lambda\ss\mu$  is 
\[\textstyle \bGQ_{\lambda \doublebar \mu} = \sum_{T\in \ShSVT_Q(\lambda\ss\mu)} \beta^{|T|-|\lambda/\mu|} x^T
\in \RR\llbracket \bfx\rrbracket.
\]
Similarly, the
 \defn{$K$-theoretic Schur $P$-function} of shape $\lambda\ss\mu$ is
\be\textstyle \bGP_{\lambda \doublebar \mu} = \sum_{T\in \ShSVT_P(\lambda\ss\mu)} \beta^{|T|-|\lambda/\mu|} x^T 
 \in \RR\llbracket \bfx\rrbracket.
\ee
Let 
$\bGP_{\lambda}=\bGP_{\lambda \doublebar \emptyset}
$
and
$
\bGQ_{\lambda}=\bGQ_{\lambda \doublebar \emptyset}
$
and define
$\bGP_{\lambda \ss\mu}=\bGQ_{\lambda \ss\mu}=0$ when $\mu\not\subseteq\lambda$.
\end{definition}

We abbreviate by setting 
\[\bG_0 = \bGP_0 =\bGQ_0 = 1,
\quad \bG_n=\bG_{(n)}, \quad\bGP_n=\bGP_{(n)}, \quand\bGQ_n=\bGQ_{(n)}\] for $n \in \PP$.
Notice that 
$
\textstyle
\bG_1 = \bGP_1 = \sum_{n\geq 1} e_n = E(1)-1.
$
One also has \cite{ChiuMarberg}
\be\label{GQ-GP-n-eq}
\bGQ_n = 2\bGP_n + \beta \bGP_{n+1} \quad\text{for all }n \in \PP.
\ee
In general,  $\bGQ_\lambda$ is an $\RR$-linear but not always an $\RR_{\geq0}$-linear combination of $\bGP$-functions \cite{ChiuMarberg}.

\begin{example}\label{sh-ss-ex}
If $0< m\neq n$ then   $\bGQ_{(n) \doublebar (m)} = \bGQ_{n-m} + \beta \bGQ_{n-m+1}$
 while 
\[ 
\bGP_{(n) \doublebar (m)}= \bGQ_{n-m} + \begin{cases}
  \beta \bGQ_{n-m+1}  & \text{if }m>1 \\
  \beta \bGP_{n}  & \text{if }m=1.
  \end{cases}
  \]
\end{example}

 As a variant of the above constructions, for any strict partitions $\mu\subseteq\lambda$ let 
\[\ShSVT_P(\lambda/\mu)\subset \ShSVT_P(\lambda\ss\mu)
\quand \ShSVT_Q(\lambda/\mu)\subset \ShSVT_Q(\lambda\ss\mu)
\]
be the subsets of tableaux $T$ that have $T_{ij}=\varnothing$ for all $(i,j)\in \SD_\mu$.

\begin{definition}[\cite{LM}]
The \defn{$K$-theoretic Schur $P$- and $Q$-functions} of shape  $\lambda/\mu$ are
\[
\bGP_{\lambda / \mu} = \sum_{T\in \ShSVT_P(\lambda/\mu)} \beta^{|T|-|\lambda/\mu|} x^T
\quand
\bGQ_{\lambda / \mu} = \sum_{T\in \ShSVT_Q(\lambda/\mu)} \beta^{|T|-|\lambda/\mu|} x^T.
\]
\end{definition}

These power series are related to $\bGQ_{\lambda \ss\mu}$ and $\bGP_{\lambda \ss\mu}$ 
by the identities
\be\label{singlebar-eq}
\bGP_{\lambda /\mu}=
\sum_{\nu\subseteq\mu}  (-\beta)^{|\mu|-|\nu|} \bGP_{\lambda \ss \nu}
\quand
\bGQ_{\lambda /\mu}=
\sum_{\nu\subseteq\mu}  (-\beta)^{|\mu|-|\nu|} \bGQ_{\lambda \ss \nu}
\ee
where the sums are indexed by strict partitions  \cite[Cor.~5.7]{LM}.

\subsection{Unbounded symmetric functions}\label{unbounded-sect}

Let $\mSym\subset \RR\llbracket \bfx \rrbracket$ be the subalgebra 
of all symmetric formal power series, not necessarily of bounded degree.
Each element of $\mSym$ can be uniquely expressed as an infinite linear combination of Schur functions $s_\lambda$,
and we define $\mSym\htimes\mSym$ to be the real vector space of infinite $\RR$-linear combinations of tensor $s_\lambda \otimes s_\mu$.
This space is larger than the usual tensor product $\mSym\otimes \mSym$.

Each $\bG_{\lambda\ss\mu}$ is a finite $\RR$-linear combination of $\bG_\nu$'s \cite{Buch2002} and 
each $\bG_\lambda$ is an infinite $\RR$-linear combination of Schur functions $s_\mu$ \cite{Lenart}.
Hence, the $\RR$-vector space $\bGamma$ spanned by all $\bG_{\lambda\ss\mu}$'s has
\[
 \bGamma \subsetneq \mSym
\quand
 \bGamma\otimes \bGamma\subsetneq \mSym\htimes\mSym.
\]
Similarly, when $\lambda$ and $\mu$ are strict partitions, each  
$\bGP_{\lambda\ss\mu}$
(respectively, $\bGQ_{\lambda\ss\mu}$)
 is a finite $\RR$-linear combination of $\bGP_\nu$'s (respectively, $\bGQ_{\nu}$'s)   \cite[Prop.~4.33]{MarbergHopf}.
 In turn, each $\bGQ_\nu$ is a finite linear combination of $\bGP$-functions \cite{ChiuMarberg} 
 and each $\bGP_\nu$ is
 a finite linear combination of $\bG$-functions \cite[Thm.~3.27]{MarScr}.
 Hence the  vector  spaces $\bGammaP = \RR\spanning\left\{ \bGP_{\lambda\ss\mu}  : \text{strict }\mu\subseteq\lambda \right\}$ and $\bGammaQ= \RR\spanning\left\{ \bGQ_{\lambda\ss\mu} :\text{strict }\mu\subseteq\lambda\right\}$ from the introduction satisfy
 \[
 \bGammaQ\subsetneq \bGammaP\subsetneq \bGamma 
 \quand
 \bGammaQ\otimes \bGammaQ \subsetneq \bGammaP\otimes \bGammaP \subsetneq \mSym\htimes \mSym.
 \]
 Somewhat surprisingly, the subspaces  $\bGammaQ\subsetneq \bGammaP\subsetneq \bGamma$
 are subalgebras of $\mSym$ \cite{Buch2002,LM2}.

 \begin{remark}\label{rescaling-rem}
 Let $\fkR : \RR\llbracket\bfx\rrbracket \to \RR\llbracket\bfx\rrbracket$ be the rescaling morphism $f \mapsto f(|\beta| x_1, |\beta|x_2, |\beta|x_3,\dots)$.
 \ben
 \item[(a)] First assume 
  $\beta >0$. Then we have
 \[
 \fkR(\pG_{\lambda\ss\mu}) = \beta^{|\lambda|-|\mu|} \bG_{\lambda\ss\mu},
 \ \
  \fkR(\pGP_{\lambda\ss\mu}) = \beta^{|\lambda|-|\mu|} \bGP_{\lambda\ss\mu},
  \ \ 
   \fkR(\pGQ_{\lambda\ss\mu}) = \beta^{|\lambda|-|\mu|} \bGQ_{\lambda\ss\mu},
 \]
and so $\fkR$ restricts to algebra isomorphisms $\pGamma\xrightarrow{\sim}\bGamma$ and
$\pGammaP\xrightarrow{\sim}\bGammaP$ and $\pGammaQ\xrightarrow{\sim}\bGammaQ$.
Composition with $\fkR$ gives a bijection from
$\bG$-  to $\pG$-positive specializations,
as well as from
$\bGP$-  to $\pGP$-positive specializations 
and 
from 
$\bGQ$-  to $\pGQ$-positive specializations.

\item[(b)] Similarly, when $\beta <0$ the map
%
$\fkR$ restricts to algebra isomorphisms $\mGamma\xrightarrow{\sim}\bGamma$ and
$\mGammaP\xrightarrow{\sim}\bGammaP$ and $\mGammaQ\xrightarrow{\sim}\bGammaQ$,
and composition with $\fkR$ gives a bijection from
$\bG$-, $\bGP$-, and $\bGQ$-positive specializations  to $\mG$-, $\mGP$-, and $\mGQ$-positive specializations

\item[(c)] Finally if $\beta=0$ then $\bGamma= \Sym$ and $\bGammaP = \bGammaQ =\SymSh$ from Section~\ref{nazar-sect}.
\een
Thus, to classify all $\bG$-, $\bGP$-, and $\bGQ$-positive specializations
of our three $\RR$-algebras, one just needs to understand the classical case when $\beta=0$ and  the two other cases when $\beta =\pm 1$.
 \end{remark}

A linear map $f$ with domain $\mSym$ is \defn{continuous} if $f\(\sum_\lambda c_\lambda s_\lambda\) = \sum_\lambda c_\lambda f(s_\lambda)$ for all $c_\lambda \in \RR$.
The counit and coproduct of $\Sym$ extend to continuous linear maps $\mSym \to \RR$ and $\mSym\to \mSym\htimes\mSym$.
Restricting these makes $\bGamma$ into a cocommutative bialgebra \cite[Cor.~6.7]{Buch2002} with
\be
  \Delta(\bG_{\lambda\doublebar\mu}) = \sum_{\mu\subseteq \kappa\subseteq \lambda} \bG_{\kappa\doublebar\mu} \otimes \bG_{\lambda\doublebar\kappa}.
  \ee
Additionally, the subspaces $\bGammaP$ and $\bGammaQ$ are sub-bialgebras of $\bGamma$
with coproduct formulas 
\be\label{delta-eq}
  \Delta(\bGP_{\lambda\doublebar\mu}) = \sum_{\mu\subseteq \kappa\subseteq \lambda} \bGP_{\kappa\doublebar\mu} \otimes \bGP_{\lambda\doublebar\kappa} \quand
    \Delta(\bGQ_{\lambda\doublebar\mu}) = \sum_{\mu\subseteq \kappa\subseteq \lambda} \bGQ_{\kappa\doublebar\mu} \otimes \bGQ_{\lambda\doublebar\kappa}
  \ee
  where the sums are indexed by strict partitions   \cite[Prop.~4.33]{MarbergHopf}.

  \begin{remark}\label{semigroup-rem}
     In view of the above formulas, the generating sets $\{ \bG_{\lambda\ss\mu}\}$, $\{ \bGP_{\lambda\ss\mu}\}$, and $\{ \bGQ_{\lambda\ss\mu}\}$ are comultiplicative subsets of $\bGamma$, $\bGammaP$, and $\bGammaQ$.
 Hence, the sets of $\bG$-, $\bGP$-, and $\bGQ$-positive specializations for each bialgebra
 form commutative semigroups under the union operation.
 \end{remark}
 
We mention two other notable facts.
First, each $\bGP_\lambda$ and $\bGQ_\lambda$ is a $\NN[\beta]$-linear combination of $\bG_\nu$'s \cite[Thms.~3.27 and 3.40]{MarScr}.
Second,
the multiplicative and comultiplicative structure constants for the three bialgebras $\bGammaQ\subsetneq \bGammaP\subsetneq\bGamma$
 all belong to $\NN[\beta]$;
see \cite[Cors.~5.5 and 6.7]{Buch2002} for $\bGamma$ and \cite[Thm.~1.6]{LM2} and \cite[Cor.~4.42 and Rem.~4.43]{Mar2025}
for $\bGammaP$ and $\bGammaQ$.
This implies:

 \begin{proposition}\label{ppp-prop}
 Suppose  $\beta\geq 0$ and 
 $f$, $g$, and $h$ are specializations of $\bG$, $\bGP$, and $\bGQ$.
Then these specializations are respectively $\bG$-, $\bGP$-, and $\bGQ$-positive 
if and only if 
\[f(\bG_\lambda)\geq 0,
\quad 
g(\bGP_\mu)\geq 0,
\quand 
h(\bGQ_\mu)\geq 0
\]
for all partitions $\lambda$ and all strict partitions $\mu$.
Moreover, if $f$ is $\bG$-positive then its restrictions to $\bGammaP$ and $\bGammaQ$ are respectively $\bGP$- and $\bGQ$-positive.
%
%
%
 \end{proposition}

  The involution $\omega$ from \eqref{omega-def}   extends to a continuous linear map $\mSym\to\mSym$,
which restricts to an algebra morphism (but not an automorphism) $\bGamma \to \mSym$ satisfying \cite[Thm.~6.2]{Yel19}
\be\label{omega-eq}
\omega(\bG_{\lambda\doublebar\mu}) = \bG_{\lambda^\top\doublebar\mu^\top}(\tfrac{x_1}{1-\beta x_1}, \tfrac{x_2}{1-\beta x_2}, \tfrac{x_3}{1-\beta x_3},\dots)
\ee
where $\lambda^\top$ is the transpose of a partition $\lambda$.
In $\bGammaP$ and $\bGammaQ$ we have the simpler formulas 
\be\label{omega-pq-eq}
\omega(\bGP_{\lambda\doublebar\mu}) = \bGP_{\lambda\doublebar\mu}(\tfrac{x_1}{1-\beta x_1}, \tfrac{x_2}{1-\beta x_2},\dots)
\quand
\omega(\bGQ_{\lambda\doublebar\mu}) = \bGQ_{\lambda\doublebar\mu}(\tfrac{x_1}{1-\beta x_1}, \tfrac{x_2}{1-\beta x_2},\dots)
\ee
which do not involve any change in the strict partition indices \cite[Prop.~5.6 and Cor.~6.6]{LM}.
 
  Let $\bOmega : \bGamma \to \bGamma$ be the linear map with 
  \be \bOmega(\bG_\lambda) = \bG_{\lambda^\top}\quad\text{for all partitions $\lambda$}.\ee
Then let $\bPhi $ and $\bPsi$ be the maps $\mSym \to \mSym$ with the formulas
\be
\bPhi(f) = f(\tfrac{x_1}{1+\beta x_1},\tfrac{x_2}{1+\beta x_2},\tfrac{x_3}{1+\beta x_3},\dots)
\quand
\bPsi(f) = f(\tfrac{x_1}{1-\beta x_1},\tfrac{x_2}{1-\beta x_2},\tfrac{x_3}{1-\beta x_3},\dots).
\ee

 \begin{lemma}\label{bOmega-lem}
The map $\bOmega$
 is a bialgebra involution of $\bGamma$ with $\bOmega = \bPhi \circ \omega = \omega \circ \bPsi$ and
 \[
\bOmega(\bG_{\lambda\doublebar\mu}) = \bG_{\lambda^\top\doublebar\mu^\top},
\quad
\bOmega(\bGP_{\nu\doublebar\kappa}) = \bGP_{\nu\doublebar\kappa},
\quand
\bOmega(\bGQ_{\nu\doublebar\kappa}) = \bGQ_{\nu\doublebar\kappa}
\]
for all partitions $\lambda$, $\mu$ and all strict partitions $\nu$, $\kappa$.
 \end{lemma}
 
 \begin{proof}
The identities \eqref{omega-eq} and \eqref{omega-pq-eq} imply that $\bOmega = \bPhi \circ \omega$  is  a bialgebra morphism,
and that the last three formulas hold.
Since $\omega$ is an involution we have  $( \bPhi \circ \omega) \circ (\omega \circ \bPsi) =(\omega \circ \bPsi)\circ ( \bPhi \circ \omega) = 1$,
so as $\bOmega$ is also an involution it must hold that $\bOmega = \bPhi \circ \omega = \omega \circ \bPsi$.
 \end{proof}

 The multiplication formulas in \cite{Lenart} (restated as \cite[Eq.~(5)]{Yel20}) imply that
\be\label{G-pieri}
\bG_1 \bG_\mu = \sum_{\lambda} \beta^{|\lambda|-|\mu|} \bG_\lambda
\ee
where the sum is over all partitions $\lambda$ with $\mu\subsetneq\lambda$ such that $\D_\lambda\setminus\D_\mu$ 
is a \defn{rook strip} in the sense of containing at most one box in each row and column. For example,
\[ \bG_1 \bG_{(3,1)} = \bG_{(4,1)} + \bG_{(3,2)} + \bG_{(3,1,1)} + \beta \bG_{(4,2)} + \beta \bG_{(4,1,1)}+ \beta \bG_{(3,2,1)}+\beta^2  \bG_{(4,2,1)}.\]
Similarly, if $\mu$ is a strict partition then \cite[Cor.~4.8]{BuchRavikumar} implies that
\be\label{GP-pieri}
\bGP_1 \bGP_\mu = \sum_{\lambda} \beta^{|\lambda|-|\mu|} \bGP_\lambda
\ee
where the sum is over all strict partitions $\lambda\supsetneq \mu$ such that $\SD_\lambda\setminus\SD_\mu$ 
is a  rook strip. For example,
\[ \bGP_1 \bGP_{(3,1)} = \bGP_{(4,1)} + \bGP_{(3,2)}  + \beta \bGP_{(4,2)}.
\]
We also have 
 \be
 (\bGP_1)^2 = \bGP_2
 \qquand \(1+\beta\bGP_1\)^2 = 1+\bGQ_1.\ee
The rule to expand the product of $\bGQ_1$ and $\bGQ_\lambda$ in the $\bGQ$-basis is more complicated \cite[Cor.~5.6]{BuchRavikumar}.
However, the formulas in \cite{BuchRavikumar,Lenart} imply
that $\bGamma$, $\bGammaP$, and $\bGammaQ$ are generated as $\RR$-algebras
by the respective sets of elements $\{ \bG_n : n \in \PP\}$,
$\{ \bGP_n : n \in \PP\}$,
and
$\{ \bGQ_n : n \in \PP\}$.

\subsection{Unsigned specializations}\label{unsigned-sect}

We now consider the bialgebras $\pGammaQ\subsetneq \pGammaP\subsetneq \pGamma$
obtained from $\bGammaQ\subsetneq \bGammaP\subsetneq \bGamma$
by setting $\beta=1$.
This section contains 
the proof of Theorem~\ref{main-thm} from the introduction.
Before presenting this, we require some preliminaries on the classification of 
the $\pG$-positive specializations of $\pGamma$ from \cite{Yel20}.

Let $a=(a_1\geq a_2 \geq a_3\geq \dots \geq 0)$ be a sequence of nonnegative real numbers with finite sum. 
Yeliussizov \cite[Lem.~4.11]{Yel20} has shown that the  map $\phi_a : \Sym \to \RR$ 
extends to a $\pG$-positive specialization of $\pGamma$ in the sense that if we write 
$
\pG_\lambda = \sum_{\mu\supseteq\lambda}  r_{\mu/\lambda} s_\mu
$ for integers $r_{\mu/\lambda}\in \NN$ (which is possible by results in \cite{Lenart}), then
the sum
$
 \sum_{\mu\supseteq\lambda}  r_{\mu/\lambda} \phi_a(s_\mu)
$
always
converges. 

The map
 $\pi_\gamma : \Sym \to \RR$ similarly extends to a $\pG$-positive specialization of $\pGamma$
 for any $\gamma \in \RR_{\geq 0}$,
 while the composition $\varepsilon_a=\phi_a \circ \omega$ extends to a well-defined algebra morphism $\pGamma\to\RR$
 (which is then another $\pG$-positive specialization)
  if and only if $\max(a)<1$
 \cite[Lem.~4.11]{Yel20}. 
 
 Following \cite{Yel20}, we write these respective extensions as
\be
\wphi_a : \pGamma\to\RR,\quad
\wpi_\gamma : \pGamma\to\RR,
\quand
\wvarepsilon_a : \pGamma\to\RR
\ee
We now present an alternate form of Theorem~\ref{y-thm1} from the introduction.

\begin{theorem}[\cite{Yel20}]
\label{y-thm3}
The $\pG$-positive specializations of $\pGamma$ are the maps  
$\rho = \wphi_\alpha \sqcup \wvarepsilon_b \sqcup \wpi_\gamma$
where  $\gamma\in\RR_{\geq 0}$ and  $a=(a_1\geq a_2 \geq \dots \geq 0)$ and $b=(1>b_1\geq b_2\geq  \dots \geq 0)$
are real sequences with
\[ \textstyle \prod_{n=1}^\infty \frac{1+a_n}{1-b_n}\in \RR.\]
The representation of $\rho = \wphi_\alpha \sqcup \wvarepsilon_b \sqcup \wpi_\gamma$
is unique and it holds that $\textstyle \rho (\pG_1)= -1 + e^\gamma \prod_{n=1}^\infty \frac{1+a_n}{1-b_n} .$

\end{theorem}

\begin{remark}
The preceding result is equivalent to Theorem~\ref{y-thm1} since by \cite[Lem.~4.7]{Yel20} we have
\be
\label{GG-eq}
\textstyle
\sum_{n=0}^\infty \(\bG_n +\beta \bG_{n+1}\)z^n
  = \(1+\beta \bG_1\)H(z).
  \ee
Using this identity and \eqref{12cup-eq0}, \eqref{12cup-eq00}, and \eqref{12cup-eq}, one checks that if 
    $\rho = \wphi_\alpha \sqcup \wvarepsilon_b \sqcup \wpi_\gamma$ then
        \eqref{y-thm1-eq} holds
    with $1+\rho(\pG_1)=C<\infty$ as in \eqref{C-as-eq}.
Conversely, if these identities hold for a specialization $\rho : \pGamma\to \RR$ then we must have 
 $\rho = \wphi_\alpha \sqcup \wvarepsilon_b \sqcup \wpi_\gamma$  since the
 coefficients of the power series $   \sum_{n\geq 0} \rho(\pG_n+\pG_{n+1}) z^{n}$
  determine the values of $\rho$ on the set of algebra generators $\{\pG_n: n \in \PP\}\subset \pGamma$.
      
      The uniqueness of the parameters $a$, $b$, and $\gamma$ follows by considering the zeros and poles of 
 \eqref{y-thm1-eq}.
    Finally, we have $e^\gamma \prod_{n=1}^\infty \frac{1+a_n}{1-b_n} \in [1,\infty)$ 
    if and only if $ \prod_{n=1}^\infty \frac{1+a_n}{1-b_n}$ converges to a finite value.
\end{remark}

We wish to use Yeliussizov's results to find a similar classification all $\pGP$- and $\pGQ$-positive specializations of $\pGammaP$ and $\pGammaQ$. This requires some nontrivial lemmas.

First,
let $\lambda$ be a partition with exactly $n$ nonzero parts.
Write $\delta = (n,\dots,3,2,1)$ for the $n$-part \defn{staircase partition}
so that $\lambda+\delta = (\lambda_1+n, \lambda_2+n-1,\dots,\lambda_n+1)$ is a strict partition.
Then there exists a unique linear map 
$\bTheta : \bGamma\to\bGammaQ$ with the formula
\be\label{bTheta-eq}
\bTheta(\bG_\lambda) =\bGP_{(\lambda+\delta)/\delta} =\bGQ_{(\lambda+\delta)/\delta}.
\ee
In fact, this map is a bialgebra morphism \cite[Cor.~5.16]{LM} and 
it is evident   that
\be
\bTheta(\bG_n) = \bGQ_n\quad\text{for all }n \in \NN.
\ee
The following less obvious lemmas rely on some theorems recently proved in \cite{Mar2025}.

\begin{lemma} \label{theta-lem}
If $\mu\subseteq\lambda$ are any partitions then 
\[\bTheta(\bG_{\lambda\ss\mu})
\in
 \NN[\beta]\spanning\left\{ \bGP_\nu : \nu\text{ strict}\right\}
 \cap
  \NN[\beta]\spanning\left\{ \bGQ_\nu : \nu\text{ strict}\right\}.
\]
\end{lemma}

\begin{proof}
As each $\bG_{\lambda\ss\mu}$ is a $\NN[\beta]$-linear combination of $\bG_\nu$'s \cite{Buch2002},
we may assume $\mu=\emptyset$.
Then it suffices to show that $\bGP_{(\lambda+\delta)/\delta}=\bGQ_{(\lambda+\delta)/\delta}$ is both a $\NN[\beta]$-linear combination of $\bGP$-functions
and  
 a $\NN[\beta]$-linear combination of $\bGQ$-functions. This is precisely \cite[Cor.~4.44]{Mar2025}.
%
\end{proof}

\begin{lemma}\label{theta-lem2}
We have $\bTheta\circ\bOmega=\bTheta$ as maps $ \bGamma\to\bGammaQ$.
\end{lemma}

\begin{proof}
Since $\bTheta$ and $\bOmega$ are bialgebra morphism,
it suffices to check this identity on the generators $\bG_n$ for $\bGamma$.
Fix $n \in \PP$.
Thus, we just need to show that   
\[
\bGQ_{(n+1,\dots,4,3,2)/(n,\dots,3,2,1)} = \bTheta(\bG_{(1^n)}) =  \bTheta\circ\bOmega(\bG_n)
\quad\text{is equal to}\quad\bTheta(\bG_n) = \bGQ_n.
\]
Here is a weight-preserving bijection $\ShSVT_Q((n))\to \ShSVT_Q((n+1,\dots,4,3,2)/(n,\dots,3,2,1))$
that realizes this identity.
Choose $T \in \ShSVT_Q((n))$. For each $i \in \PP$, let $i' = i-\frac{1}{2}$
and consider the boxes of $T$ containing $i$ or $i'$ (or both).
These boxes must be adjacent and of the form
\[ \ytabc{0.6cm}{i & i & i& \cdots & i &  i }
\quord 
\ytabc{0.6cm}{i' & i & i&\cdots & i&  i }
\quord 
\ytabc{0.6cm}{i'i & i & i &\cdots &i& i }
\] 
when we ignore all entries except $i$ or $i'$. Respectively change these patterns to
\[ \ytabc{0.6cm}{i' & i' & i'& \cdots &i' & i' }
\quord 
\ytabc{0.6cm}{i' & i' & i'&\cdots & i' & i }
\quord 
\ytabc{0.6cm}{i' & i' & i' &\cdots & i' &i'i }.
\] 
After performing this operation for all $i \in \PP$, transpose the boxes of $T$ to obtain a tableau $U$
with $n$ boxes arranged in one column. The desired bijection is then $T\mapsto U$.
\end{proof}

Our final lemma in this section is a shifted version of \eqref{GG-eq}. Here, we define 
\be\label{x-bar-eq} \overline{x} = \frac{-x}{1+x}\quad\text{for any parameter $x$}.
\ee

\begin{lemma}[\cite{HIMN}]
\label{HIMN-lem}
It holds that \[
 \sum_{n=0}^\infty\( \bGQ_{n} +\beta \bGQ_{n+1}\) z^{n} 
= \(1+\beta \bG_1\)E(z+\beta )H(z)
.\]
\end{lemma}

\begin{proof}
After unpacking the notation in \cite[\S10]{HIMN}, this 
 identity follows from \cite[Cor.~10.10]{HIMN}. 
\end{proof}

The preceding lemma can also be shown by a more direct combinatorial argument.
In any event, 
this brings us to our shifted version of Theorem~\ref{y-thm3}.

\begin{theorem}\label{main-thm2}
The $\pGP$-positive (respectively, $\pGQ$-positive) specializations of $\pGammaP$ (respectively, $\pGammaQ$) are the restrictions of
the maps  
$\rho = \wphi_\alpha  \sqcup \wpi_\gamma$
where $\gamma\in \RR_{\geq 0}$ and where $a=(a_1\geq a_2 \geq \dots \geq 0)$ is a real sequence
with 
$ \prod_{n=1}^\infty(1+a_n) \in \RR$.
For such maps  one has  
\be\label{main-thm2-eq}
\sum_{n\geq 0} \rho(\pGQ_n+\pGQ_{n+1}) z^{n} 
 = D^2 e^{2\gamma z} \prod_{n=1}^\infty \frac{1-\overline{a_n} z}{1-a_n z}
 \ee
where $D = 1+\rho(\pG_1) = e^{\gamma}\prod_{n=1}^\infty(1+a_n)$ and $D^2 = 1+\rho(\pGQ_1)$.
\end{theorem}

\begin{proof}
In view of Proposition~\ref{ppp-prop}, any such map $\rho = \wphi_\alpha  \sqcup \wpi_\gamma$ restricts to $\pGP$- and $\pGQ$-positive specializations of $\pGammaP$ of $\pGammaQ$, and we have $ \rho (\pGP_1)=\rho (\pG_1)= -1+ e^{\gamma} \prod_{n=1}^\infty (1+a_n)$ by Theorem~\ref{y-thm3}.
Then by
using Lemma~\ref{HIMN-lem} with \eqref{12cup-eq0}, \eqref{12cup-eq00}, and \eqref{12cup-eq}, one can check that \eqref{main-thm2-eq} holds.

Now suppose
$\rho$ is a $\pGP$-positive specialization of $\pGammaP$ or a $\pGQ$-positive specialization of $\pGammaQ$. Let $\Theta = \Theta^{(1)}$ and $\Omega=\Omega^{(1)}$.
Then  Lemmas~\ref{theta-lem} and \ref{theta-lem2}  imply that
 $\rho \circ \Theta = \rho\circ\Theta\circ \Omega$ is a $\mG$-positive specialization of $\mGamma$,
 so  
 $\rho \circ \Theta =  \wphi_a \sqcup \wvarepsilon_b \sqcup \wpi_\gamma$ for unique parameters as in Theorem~\ref{y-thm3}.
 
 As $\Omega$ is a bialgebra morphism,  we have
 \[
 (\wphi_a \sqcup \wvarepsilon_b \sqcup \wpi_\gamma) \circ \Omega  =
( \wphi_a\circ \Omega) \sqcup (\wvarepsilon_b\circ \Omega) \sqcup (\wpi_\gamma\circ \Omega).
\]
The formulas \eqref{limit-eq}, \eqref{pi-omega-eq}, and $\Omega = \Phi^{(1)}\circ \omega = \omega \circ \Psi^{(1)}$ in Lemma~\ref{bOmega-lem}
imply that 
\[ \wphi_a\circ \Omega = \wvarepsilon_{a'},
\quad 
 \wvarepsilon_b\circ \Omega = \wphi_{b'},
 \quand
 \wpi_\gamma \circ \Omega = \wpi_\gamma
 \]
 for $a' = (\frac{a_1}{1+a_1}, \frac{a_2}{1+a_2},\dots)$ and $b'=(\frac{b_1}{1-b_1} , \frac{b_2}{1-b_2},\dots)$.
Since  $\wphi_a \sqcup \wvarepsilon_b \sqcup \wpi_\gamma = (\wphi_a \sqcup \wvarepsilon_b \sqcup \wpi_\gamma) \circ \Omega $,
the uniqueness of our parameters means 
that we must have $a=b'$ and $b=a'$. 

Thus, Theorems~\ref{y-thm1} and \ref{y-thm3} imply that
\be\label{mer-eq}
  \sum_{n\geq 0} \rho(\pGQ_n+\pGQ_{n+1}) z^{n} 
  =  \sum_{n\geq 0} \rho \circ \Theta(\pG_n + \pG_{n+1}) z^n
 = (1+\rho(\pGQ_1)) e^{\gamma z} \prod_{n=1}^\infty \frac{1-\overline{a_n} z}{1-a_n z}
 \ee
where we have
 \[
 1+\rho(\pGQ_1) = 1+\rho\circ \Theta(\bG_1) =   e^\gamma \prod_{n=1}^\infty \frac{1+a_n}{1+\overline{a_n}}
=e^\gamma \prod_{n=1}^\infty (1+a_n)^2,
\]
so the infinite product $ \prod_{n=1}^\infty (1+a_n)$ must converge.
Again using Lemma~\ref{HIMN-lem} with \eqref{12cup-eq0}, \eqref{12cup-eq00}, and \eqref{12cup-eq}, 
one checks that the right side of \eqref{mer-eq} is equal to
$
 \sum_{n\geq 0}  (\wphi_a  \sqcup \wpi_{\gamma/2})(\pGQ_n+\pGQ_{n+1}) z^{n}.
$
Thus   $\rho $ must coincide with the restriction of $\wphi_\alpha  \sqcup \wpi_{\gamma/2}$
since
the coefficients of 
\[ \sum_{n\geq 0} \rho(\pGQ_n+\pGQ_{n+1}) z^{n} = (1 + \rho(\pGP_1))^2 + \sum_{n\geq 1} \rho(2\pGP_n+3\pGP_{n+1}+\pGP_{n+2}) z^{n} 
\]
determine the values of $\rho$ on the relevant set of algebra generators $\{\pGP_n\} \subset \pGammaP$ or $\{\pGQ_n\} \subset \pGammaQ$
(as $1 + \rho(\pGP_1)$ is positive and $\pGP_2= (\pGP_1)^2$).
Replacing $\gamma$ by $2\gamma$ gives the desired form of $\rho$.
\end{proof}

\begin{proof}[Proof of Theorem~\ref{main-thm}]
In the proof of Theorem~\ref{main-thm2}, we showed that the 
$\pGP$-positive specializations of $\pGammaP$ and the $\pGQ$-positive specializations of $\pGammaQ$
are obtained by restricting the same $\pG$-positive specializations $\rho$ of $\pGamma$, 
and that \eqref{main-thm2-eq} holds for each such specialization.
We also observed in our argument that if \eqref{main-thm2-eq} holds then $\rho$ coincides with the restriction of $\wphi_a  \sqcup \wpi_\gamma$ for appropriate parameters $a$ and $\gamma$, so $\rho$ restricts to 
$\pGP$- and $\pGQ$-positive specializations by Proposition~\ref{ppp-prop} and Theorem~\ref{y-thm3}.
This confirms the equivalence of   (a), (b), and (c) in  Theorem~\ref{main-thm}.
\end{proof}

\subsection{Single-variable formulas}
 
 The results in this section, which are needed to classify the $\mGP$- and $\mGQ$-positive specializations of $\mGammaP$ and $\mGammaQ$,
 are $\bGP$- and $\bGQ$-analogues of some $\bG$-function formulas in \cite{Yel19}. 
 
 Let $\mu\subseteq \lambda$ be strict partitions such
that the shifted skew shape $\SD_{\lambda/\mu} := \SD_\lambda\setminus\SD_\mu$ is a \defn{border strip} in the sense that $(i,j) \in \SD_{\lambda/\mu}$ implies that $(i+1,j+1) \notin \SD_{\lambda/\mu}$. Define 
\[
\corners(\lambda,\mu) = \{ (i,j) \in \SD_\mu : (i+1,j)\notin \SD_\mu,\ (i,j+1) \notin \SD_\mu,\ (i+1,j+1) \notin \SD_\lambda\}.
\]
Then let 
\[
\ba
\free_Q(\lambda,\mu) &= \{ (i,j) \in \SD_{\lambda/\mu} : (i+1,j),(i,j-1) \notin \SD_{\lambda/\mu} \}, 
\\
 \free_P(\lambda,\mu) &= \{ (i,j) \in \free_Q(\lambda,\mu) : i \neq j\},
  \ea
 \]
along with
\[
\ba
  \gaps_Q(\lambda,\mu) &= \{ (i,j) \in \corners(\lambda,\mu) : (i+1,j),(i,j+1)\in \SD_{\lambda/\mu} \}, \\
   \gaps_P(\lambda,\mu) &=\gaps_Q(\lambda,\mu) \sqcup \{ (i,j) \in \corners(\lambda,\mu) : i=j\text{ and }(i+1,j)\in \SD_{\lambda/\mu} \}.
  \ea
 \]
 Finally let
 \[
\ba
  \exts_Q(\lambda,\mu) &= \{ (i,j) \in \corners(\lambda,\mu) : (i+1,j)\in \SD_{\lambda/\mu}\text{ or }(i,j+1)\in \SD_{\lambda/\mu} \}- \gaps_Q(\lambda,\mu), \\
    \exts_P(\lambda,\mu) &= \{ (i,j) \in \corners(\lambda,\mu) : (i+1,j)\in \SD_{\lambda/\mu}\text{ or }(i,j+1)\in \SD_{\lambda/\mu}\text{ or }i=j \}- \gaps_P(\lambda,\mu).
  \ea
 \]
 When $f \in \RR\llbracket \bfx \rrbracket$ is a power series and $a_1,a_2,\dots,a_k$ is a finite list of parameters,
 we let
 \be f(a_1,a_2,\dots,a_k) = f(a_1,a_2,\dots,a_k,0,0,0,\dots)\ee
denote the result of substituting $x_i\mapsto a_i$ for $1\leq i \leq k$ and $x_i \mapsto 0$ for $i>k$.
\begin{proposition}\label{1var-prop}
Let $\lambda $ and $\mu$ be strict partitions.
If $\mu\not\subseteq \lambda$ or if $\SD_{\lambda/\mu}$ is not a border strip, then 
the single-variable power series $\bGP_{\lambda\doublebar\mu}(x)=\bGQ_{\lambda\doublebar\mu}(x)=0$ are both zero.
Otherwise, we have 
\[\ba
\bGP_{\lambda\doublebar\mu}(x)  &= 
2^{a_P}  (2 +\beta x)^{b_P} (1 +\beta x)^{c_P} x^{|\lambda|-|\mu|} \quand
\\
\bGQ_{\lambda\doublebar\mu}(x)  &= 
2^{a_Q} (2 +\beta x)^{b_Q} (1 +\beta x)^{c_Q} x^{|\lambda|-|\mu|}
\ea
\]
where for either type $K \in \{P,Q\}$ we define 
\[ \ba
a_K=a_K(\lambda,\mu) &= |\gaps_K(\lambda,\mu)|,
\\
b_K=b_K(\lambda,\mu)   &= |\free_K(\lambda,\mu)|-|\gaps_K(\lambda,\mu)|,
\\
c_K=c_K(\lambda,\mu)   &= 2|\corners(\lambda,\mu)| - |\gaps_K(\lambda,\mu)| - |\exts_K(\lambda,\mu)|
.
\ea
\]
\end{proposition}

Before giving a proof of this result, we note a corollary and present some examples.

\begin{corollary}
If $\lambda $ is a strict partition with more than one part then  $\bGP_{\lambda}(x)=\bGQ_{\lambda}(x)=0$.
Otherwise, if $n \in \PP$ then  
$\bGP_{n}(x)  =   x^{n}$
and
$
\bGQ_{n}(x)  = 
(2 +\beta x) x^{n}$.
\end{corollary}

 \begin{example}\label{1var-ex1}
 Suppose $\mu = (2)$ and $\lambda=(5)$. Then 
  \[
 \SD_{\lambda/\mu} = 
 \ytabc{.4cm}{  *(lightgray) & *(lightgray) & \ & \ & \
 }
\]
so we have 
 \[
 \ba
 \SD_{\lambda/\mu} &= \{(1,3),(1,4),(1,5)\}, \\  \free_P(\lambda,\mu) = \free_Q(\lambda,\mu)&=\{(1,3)\},\\
\corners(\lambda,\mu)=\exts_P(\lambda,\mu)=\exts_Q(\lambda,\mu) &= \{(1,2)\}, \\
\gaps_P(\lambda,\mu) = \gaps_Q(\lambda,\mu)&=\varnothing.
\ea
\]
along with $a_P=a_Q = 0$, $b_P=b_Q = 1$, $c_P=c_P= 1$, and $|\lambda/\mu|=3$. Hence 
\[ \bGP_{\lambda\doublebar\mu}(x)  = 
\bGQ_{\lambda\doublebar\mu}(x) = (2 +\beta x) (1 +\beta x) x^3.\]
   \end{example}

    \begin{example}\label{1var-ex2}
 Suppose $\mu = (4,2)$ and $\lambda=(4,3)$. Then 
 \[
 \SD_{\lambda/\mu} = 
 \ytabc{.4cm}{
*(lightgray)  &  *(lightgray)  &  *(lightgray) &  *(lightgray)  \\
 \none &  *(lightgray)  &  *(lightgray)  & \ 
 }
\]
 so we have 
 \[
 \ba
 \SD_{\lambda/\mu} = \free_P(\lambda,\mu) = \free_Q(\lambda,\mu)&=\{(2,4)\},\\
\corners(\lambda,\mu)=\exts_P(\lambda,\mu)=\exts_Q(\lambda,\mu) &= \{(1,4),(2,3)\}, \\
\gaps_P(\lambda,\mu) = \gaps_Q(\lambda,\mu)&=\varnothing.
\ea
\]
so $a_P=a_Q = 0$, $b_P=b_Q = 1$, $c_P=c_Q= 2$, and $|\lambda/\mu|=1$. Hence 
\[ \bGP_{\lambda\doublebar\mu}(x)  = 
\bGQ_{\lambda\doublebar\mu}(x) = (2 +\beta x) (1 +\beta x)^2 x.\]
   \end{example}

\begin{proof}[Proof of Proposition~\ref{1var-prop}]
We have $\bGQ_{\lambda\doublebar\mu}(x_1) = \sum_{T} \beta^{|T|-|\lambda/\mu|} x^T$
where $T$ ranges over all set-valued shifted tableaux of shape $ \lambda\doublebar\mu$ whose entries are all subsets of $\{1',1\}$
where $1'=\frac{1}{2}$.

If $\SD_{\lambda/\mu}$ is not a border strip then there are no such $T$.
Assume $\SD_{\lambda/\mu}$ is a border strip and $T$ is one of the desired tableaux.
The nonempty boxes of $T$ consist of $\SD_{\lambda/\mu}$ plus arbitrary subsets  
\[A\subseteq \gaps_Q(\lambda,\mu), \ \ B\subseteq \exts_Q(\lambda,\mu),
\ \ \text{and}\ \ C\subseteq \diff_Q(\lambda,\mu) := \corners(\lambda,\mu) - \gaps_Q(\lambda,\mu)-\exts_Q(\lambda,\mu).\]
Together, this set of boxes is itself a border strip that is a union of zero or more connected ribbons.

Define the first box of a connected ribbon to be the one with smallest column index and largest row index.
For a fixed choice of $A$, $B$, and $C$, the desired tableaux $T$ are constructed as follows: every box that is not the first in its ribbon must contain just $1'$ or just $1$, and the choice is uniquely determined,
while the first boxes in each connected ribbon may independently contain $1'$, $1$, or both $1'$ and $1$.
For all such $T$, the number of filled boxes is $|\lambda/\mu|+|A|+|B|+|C|$ while the number of connected ribbons is $ |\free_Q(\lambda,\mu)| - |A|+|C|$. We deduce
from
these observations that
\[ \ba\bGQ_{\lambda\doublebar\mu}(x) &= \sum_{\substack{A \subseteq \gaps_Q(\lambda,\mu) \\ B\subseteq \exts_Q(\lambda,\mu) \\ C\subseteq \diff_Q(\lambda,\mu)}}
\beta^{|A|+|B|+|C|} x^{|\lambda/\mu|+|A|+|B|+|C|} (2+\beta x)^{|\free_Q(\lambda,\mu)| - |A|+|C|}
\\
&=(2+\beta x)^{b_Q}
 x^{|\lambda/\mu|} \sum_{
 A 
 } (\beta x)^{|A|}(2+\beta x)^{|\gaps_Q(\lambda,\mu)| - |A|}
 \sum_{
 B
 }
(\beta x)^{|B|} 
 \sum_{
 C
 }
(2 \beta x+\beta^2 x^2)^{|C|} 
.\ea
\] Via the binomial theorem, this expression becomes
\[ (2+\beta x)^{b_Q}  x^{|\lambda/\mu|} (2+2\beta x)^{|\gaps_Q(\lambda,\mu)|} (1+\beta x)^{|\exts_Q(\lambda,\mu)|+2|\diff_Q(\lambda,\mu)|}
= 2^{a_Q}  (2+\beta x)^{b_Q} (1+\beta x)^{c_Q} x^{|\lambda/\mu|} 
\]
as desired. The argument to derive the formula for $\bGP_{\lambda\doublebar\mu}(x)$ is similar, 
with a few adjustments to incorporate the extra requirement that $1'\notin T_{ij}$ if $i=j$.
\end{proof}

Given a strict partition $\lambda = (\lambda_1>\lambda_2>\dots\geq 0)$, let $\tilde\lambda = (\lambda_2>\lambda_3>\dots \geq 0)$.

\begin{corollary}\label{1var-cor}
If $\lambda $ and $\mu$ are strict partitions then 
\[\bGP_{\lambda\doublebar\mu}(-\tfrac{1}{\beta})=\bGQ_{\lambda\doublebar\mu}(-\tfrac{1}{\beta})=\begin{cases} (-\tfrac{1}{\beta})^{\lambda_1}&\text{if }\mu = \tilde\lambda \\ 0 &\text{otherwise}.\end{cases}\]
\end{corollary}

\begin{proof}
Proposition~\ref{1var-prop}
implies that
\[\bGP_{\lambda\doublebar\mu}(-\tfrac{1}{\beta})=\bGQ_{\lambda\doublebar\mu}(-\tfrac{1}{\beta})=0\] unless $\mu\subseteq\lambda$ and $\SD_{\lambda/\mu}$ is a border strip,
in which case  
\[\bGP_{\lambda\doublebar\mu}(-\tfrac{1}{\beta}) = 2^{a_P} \cdot 0^{c_P} \cdot (-\beta)^{|\mu|-|\lambda|}\quand\bGQ_{\lambda\doublebar\mu}(-\tfrac{1}{\beta}) = 2^{a_Q} \cdot 0^{c_Q}\cdot  (-\beta)^{|\mu|-|\lambda|}.\]
As $\gaps_K(\lambda,\mu)$ and $\exts_K(\lambda,\mu)$ are disjoint subsets of $\corners(\lambda,\mu)$, both $c_P$ and $c_Q$
are positive unless $\corners(\lambda,\mu)=\varnothing$, and then $a_P=c_Q=a_P=c_Q=0$. But $\SD_{\lambda/\mu}$ is a border strip with
 $\corners(\lambda,\mu)=\varnothing$ precisely when $\mu = \tilde\lambda$ and then $|\mu|-|\lambda| = -\lambda_1$.
\end{proof}

We are led to a shifted analogue of \cite[Lem.~6.7]{Yel20}. Recall that  $\bGP_{\lambda\doublebar\nu}=\bGQ_{\lambda\doublebar\nu}=0$ if $\nu\not\subseteq\lambda$.

\begin{corollary}\label{1xx-cor}
If $\mu\subseteq \lambda $ are strict partitions with $\lambda\neq\emptyset$ then 
\[
\bGP_{\lambda\doublebar\mu}(-\tfrac{1}{\beta},x_1,x_2,x_3,\dots)= (-\tfrac{1}{\beta})^{\lambda_1}\bGP_{\tilde \lambda\doublebar\mu}
\quand
\bGQ_{\lambda\doublebar\mu}(-\tfrac{1}{\beta},x_1,x_2,x_3,\dots)=(-\tfrac{1}{\beta})^{\lambda_1}\bGQ_{\tilde \lambda\doublebar\mu}.
\]
\end{corollary}

\begin{proof}
Notice that $\bGP_{\lambda\doublebar\mu}(-\tfrac{1}{\beta},x_1,x_2,x_3,\dots) = \sum_{\mu\subseteq \kappa\subseteq\lambda} \bGP_{\lambda \doublebar \kappa}(-\tfrac{1}{\beta}) \bGP_{\kappa\doublebar\mu} $
and then substitute the formula in Corollary~\ref{1var-cor}. The other formula is derived in a similar way.
\end{proof}

The case when $\mu=\lambda=\emptyset$ is excluded above since $\bGP_{\emptyset\doublebar\emptyset}=\bGQ_{\emptyset\doublebar\emptyset}=1$.

For any partition $\lambda$ let $\ell(\lambda) = | \{ i \in \PP : \lambda_i>0\}|$ denote its number of nonzero parts.
Repeatedly applying Corollary~\ref{1xx-cor} recovers a result of 
Nobukawa and Shimazaki
\cite[Cor.~4.1]{NobukawaShimazaki}.

\begin{corollary}[\cite{NobukawaShimazaki}]
\label{ns-cor}
If $\emptyset\neq \mu\subseteq \lambda $ are strict partitions then 
\[\ba
\bGP_{\lambda}(-\tfrac{1}{\beta},-\tfrac{1}{\beta},\dots,-\tfrac{1}{\beta})&= \bGQ_{\lambda}(-\tfrac{1}{\beta},-\tfrac{1}{\beta},\dots,-\tfrac{1}{\beta})=(-\tfrac{1}{\beta})^{|\lambda|}\quand
\\
\bGP_{\lambda\doublebar\mu}(-\tfrac{1}{\beta},-\tfrac{1}{\beta},\dots,-\tfrac{1}{\beta})&= \bGQ_{\lambda\doublebar\mu}(-\tfrac{1}{\beta},-\tfrac{1}{\beta},\dots,-\tfrac{1}{\beta})=0
\ea
\]
whenever the number of variable set to $-\frac{1}{\beta}$ respectively exceeds  $\ell(\lambda)$ and $\ell(\lambda)-\ell(\mu)$.
\end{corollary}

\subsection{Signed specializations}\label{signed-sect}

Finally, we consider the instances $\mGammaQ\subsetneq \mGammaP\subsetneq \mGamma$
of the bialgebras $\bGammaQ\subsetneq \bGammaP\subsetneq \bGamma$
with $\beta=-1$.
Our goal in this section is to 
classify the $\mGP$- and $\mGQ$-positive specializations of $\mGammaP$ and $\mGammaQ$.
We do this after reviewing Yeliussizov's classifications of the $\mG$-positive specializations of $\mGamma$. 

\begin{lemma}
If $f$, $g$, and $h$ are $\mG$-, $\mGP$-, are $\mGQ$-positive specializations of $\mGamma$, $\mGammaP$, and $\mGammaQ$,
then 
 \[\{ f(\mG_n) : n=0,1,2,\dots\}, 
 \quad \{ g(\mGP_n) : n=0,1,2,\dots\}, \quand \{ h(\mGQ_n) : n=0,1,2,\dots\}\]
 are all weakly decreasing sequences of real numbers in the interval $[0,1]$.
\end{lemma}

\begin{proof}
The claim about $f$ is \cite[Lem.~6.3]{Yel20}. The claims about $g$ and $h$ hold since we have
\[
\mGP_n - \mGP_{n+1} = \tfrac{1}{2}\(\mGQ_n - \mGP_{n+1}\)= \tfrac{1}{2}\mGP_{(n+1)\ss(1)}
\quand
\mGQ_n - \mGQ_{n+1} = \mGQ_{(n+1)\ss(1)}
\]
for all $n \in \NN$ by  Example~\ref{sh-ss-ex} and \eqref{GQ-GP-n-eq}, along with $g(\mGP_0) =g(1)= h(\mGQ_0)=h(1)=1$.
\end{proof}

Let $a=(a_1\geq a_2 \geq \dots \geq 0)$ be a sequence of nonnegative real numbers with finite sum. 
It follows from \cite[Lem.~6.5]{Yel20} that the  map $\phi_a : \Sym \to \RR$ 
extends to a specialization of $\mGamma$
if and only if $\max(a) \leq 1$. 

The same result implies that $\varepsilon_a : \Sym \to\RR$ and $\pi_\gamma : \Sym\to\RR$
extend to specializations of $\mGamma$ for all choices of $a=(a_1\geq a_2 \geq \dots \geq 0)$ and $\gamma \in \RR_{\geq 0}$.
We denote these extended specializations
 using the same symbols as in Section~\ref{unsigned-sect}, namely:
\[
\vphi_a : \pGamma\to\RR,\quad
\vvarepsilon_a : \pGamma\to\RR,\quand
\vpi_\gamma : \pGamma\to\RR.
\]
The following combines \cite[Thm.~6.4 and Prop.~6.6]{Yel20}.

\begin{theorem}[\cite{Yel20}]
\label{mG-thm}
If  $\delta \in [0,1)$ and $\rho$ is a specialization of $\mGamma$
then the following are equivalent:
\ben
\item[(a)] The map $\rho$ is  a $\mG$-positive specialization of $\mGamma$ with  $ \rho(\mG_1)= \delta$.

\item[(b)] 
It holds that $\rho = \wphi_\alpha \sqcup \wvarepsilon_b \sqcup \wpi_\gamma$
for unique sequences of real numbers  
\[ a=(1> a_1\geq a_2 \geq \dots \geq 0), \quad b=(b_1\geq b_2\geq  \dots \geq 0),
\quand \gamma\geq0
\]
satisfying
$ \delta = 1 - e^{-\gamma} \prod_{n=1}^\infty \frac{1-a_n}{1+b_n}$.

\item[(c)] One has $ \sum_{n=0}^\infty \rho\( \mG_n - \mG_{n+1}\)z^n = (1-\delta) e^{\gamma z} \prod_{n=1}^\infty \frac{1+b_n z}{1-a_nz}$
for some parameters 
 as in (b).

\een
\end{theorem}

Although any $\mG$-positive specialization of $\mGamma$
 restricts to a specialization of the subalgebras $\mGammaP$ and $\mGammaQ$,
it does not follow immediately  
that these restrictions are $\bGP$- or $\bGQ$-positive.  

\begin{lemma}\label{aneg-lem}
Let $a=(1\geq a_1\geq a_2 \geq \dots \geq 0)$ and $b=(b_1\geq b_2\geq \dots \geq 0)$ be sequences of real numbers with finite sum and let $\gamma \in \RR_{\geq 0}$.
Then $\wphi_a$, $\wvarepsilon_b$, and $\wpi_\gamma$ restrict to 
$\bGP$- and $\bGQ$-positive
specializations of  $\mGammaP$ and $\mGammaQ$.
\end{lemma}

\begin{proof}
The map $\wphi_a$ is the union of the one-variable specializations $\wphi_{a_i}$,
which restrict to $\bGQ$- and $\bGP$-positive specializations of  $\mGammaP$ and $\mGammaQ$ when $a_i \in [0,1]$
by Proposition~\ref{1var-prop} with $\beta=-1$.
Hence, the same is true of $\wphi_a$.

Next, by Lemma~\ref{bOmega-lem} with $\beta=-1$ we have 
$\wvarepsilon_b\circ \Omega^{(-1)} = \wphi_b\circ \omega \circ \omega \circ \Psi^{(-1)}=
\wphi_b\circ \Psi^{(-1)}$
which is equal to $\wphi_{c}$ for the sequence $c = (\frac{b_1}{1+b_1}, \frac{b_2}{1+b_2},\dots)$.
Since $\Omega^{(-1)}$ restricts to the identity on $\mGammaP$ and $\mGammaQ$,
the desired claim about $\wvarepsilon_b$ follows from what was already shown for $\wphi_a$.

Finally, $\wpi_\gamma$ restricts to $\bGQ$- and $\bGP$-positive specializations of  $\mGammaP$ and $\mGammaQ$
since by \eqref{limit-eq} the value of $\wpi_\gamma(f)$ 
is the limit as $N\to \infty$ of 
$ \wphi_{a}(f)$ where $a=(\gamma/N,\gamma/N,\dots,\gamma/N)\in\RR^N$.
\end{proof}

The following is now clear from Theorem~\ref{mG-thm} and \eqref{GQ-GP-n-eq}
given Remark~\ref{semigroup-rem} and Lemma~\ref{aneg-lem}.

\begin{corollary}\label{res-cor}
A $\mG$-positive specialization $\rho$ of $\mGamma$ restricts to a $\mGP$-positive specialization of $\mGammaP$ with   
$\rho(\mGP_1) = \rho(\mG_1)$
 and 
 to a $\mGQ$-positive specialization of $\mGammaQ$
with $ \rho(\mGQ_1) =1 -  (1-\rho(\mG_1))^2$.
\end{corollary}

As an alternative to \eqref{x-bar-eq}, define
 \be
 \widetilde x = \frac{-x}{1-x}\quad\text{for any parameter $x$}.
 \ee
A specialization of $\mGamma$, $\mGammaP$, or $\mGammaQ$
is \defn{normalized} if its value at $\mG_1$, $\mGP_1$, or $\mGQ_1$ is 
equal to one, and \defn{unnormalized} if this value is
strictly less than one.
Corollary~\ref{res-cor} implies that 
a specialization of $\mGamma$ is normalized
if and only if its restrictions to $\mGammaP$ and $\mGammaQ$ are normalized.

\begin{theorem} \label{main-thm3}
The unnormalized $\mGP$-positive specializations of $\mGamma_P$
and the unnormalized $\mGQ$-positive specializations of $\mGamma_Q$
are each obtained by restricting some unnormalized $\pG$-positive specialization of $\pGamma$.
If  $\delta \in [0,1)$ and $\rho$ is a specialization of $\mGamma$
then the following are equivalent:
\ben
\item[(a)] The map $\rho$ restricts to  a $\mGP$-positive specialization of $\mGammaP$ with  $ \rho(\mGP_1)= \delta$.

\item[(b)] The map $\rho$ restricts to  a $\mGQ$-positive specialization of $\mGammaQ$ with  $ \rho(\mGQ_1)= 1-(1-\delta)^2$.

\item[(c)] The restriction of $\rho$ to $\mGammaP$ coincides with  $\wphi_\alpha  \sqcup \wpi_\gamma$
for some real numbers  
\[ a=(1> a_1\geq a_2 \geq \dots \geq 0)
\quand \gamma\geq0
\quad\text{satisfying
$ \textstyle\delta = 1 - e^{-\gamma} \prod_{n=1}^\infty (1-a_n)$.}\]

\item[(d)]  One has $\ds\sum_{n=0}^\infty \rho\( \mGQ_n - \mGQ_{n+1}\)z^n = (1-\delta)^2 e^{2 \gamma z} \prod_{n=1}^\infty \frac{1-\widetilde a_n z}{1-a_nz}$
for some parameters as in (c).

\een

\end{theorem}

\begin{proof}
Parts (c) and (d) are equivalent by Lemma~\ref{HIMN-lem}
 using the identities \eqref{12cup-eq0}, \eqref{12cup-eq00}, \eqref{12cup-eq} via the argument in the proof of Theorem~\ref{main-thm2}.
In turn, part (c) implies both (a) and (b)
by Theorem~\ref{mG-thm} and Corollary~\ref{res-cor}.

Conversely, suppose
$\rho$ is an unnormalized $\mGP$-positive specialization of $\mGammaP$
or an unnormalized $\mGQ$-positive specialization of $\mGammaQ$.
In either case, it follows from \eqref{singlebar-eq} and \eqref{bTheta-eq} with $\beta=-1$  that
 $\rho \circ \Theta^{(-1)}$ is a $\mG$-positive specialization of $\mGamma$
 whose value at $\mG_1$ is 
 \[ \rho(\mGQ_1) = \rho(2\mGP_1 - \mGP_2)= \rho(2\mGP_1 - (\mGP_1)^2) = 1-(1-\rho(\mGP_1))^2 < 1 .\]
 Hence $\rho \circ \Theta^{(-1)} = \wphi_\alpha \sqcup \wvarepsilon_b \sqcup \wpi_\gamma$ for unique parameters $a$, $b$, and $\gamma$ as in 
 Theorem~\ref{mG-thm}(b).
 Since $\rho \circ \Theta^{(-1)} = \rho \circ \Theta^{(-1)}\circ \Omega^{(-1)}$ by Lemma~\ref{bOmega-lem},
 it follows by the argument in the proof of Theorem~\ref{main-thm2} (just with $\beta=1$ changed to $\beta=-1$) that $b_n = -\widetilde a_n$ for all $n \in \PP$.
Therefore
\[
\sum_{n=0}^\infty \rho\( \mGQ_n - \mGQ_{n+1}\)z^n =
\sum_{n=0}^\infty \rho\circ \Theta^{(-1)}\( \mG_n - \mG_{n+1}\)z^n =
 (1-\rho(\mGQ_1)) e^{ \gamma z} \prod_{n=1}^\infty \frac{1-\widetilde a_n z}{1-a_nz}
\]
and by Theorem~\ref{mG-thm} we have 
$\textstyle 1-\rho(\mGQ_1) = e^{-\gamma} \prod_{n=1}^\infty \frac{1-a_n}{1-\widetilde a_n}= e^{-\gamma} \prod_{n=1}^\infty (1-a_n)^2.$
This is just part (d) with $\gamma$ replaced by $\gamma/2$. Thus (a) $\Rightarrow$ (d)  and (b) $\Rightarrow$ (d), so all four properties are equivalent.
\end{proof}
 
We conclude this section with
a shifted version of \cite[Thm.~6.8]{Yel20}. 
Let $\wone = \wphi_{(1,0,0,0,\dots)}$.
Recall that if $\lambda=(\lambda_1,\lambda_2,\lambda_3,\dots)$ then $\tilde\lambda=(\lambda_2,\lambda_3,\dots)$.
If $\lambda$ is empty or has only one nonzero part, then we interpret this notation to mean $\tilde\lambda=\emptyset$.
\begin{proposition}\label{+prop}
Let $\varphi$ be a specialization of $\mGammaP$ (respectively, $\mGammaQ$).
Then $\varphi^+ := \wone \sqcup \varphi$ satisfies
\be\label{+prop-eq}
\varphi^+(\mGP_{\lambda\doublebar\mu}) = \varphi(\mGP_{\tilde\lambda\doublebar\mu})
\quand
\varphi^+(\mGQ_{\lambda\doublebar\mu}) = \varphi(\mGQ_{\tilde\lambda\doublebar\mu})
\ee
for all strict partitions $\lambda$ and $\mu$. Consequently, it holds that:
\ben
\item[(a)] $\varphi^+(\mGP_{n})=1$ and $\varphi^+(\mGQ_{n}) = 1$ for all $n\in\PP$.
\item[(b)] $\varphi^+$ is a $\mGP$- or $\mGQ$-positive specialization if and only if $\varphi$ has the same property.
\een
\end{proposition}

\begin{proof}
The identities \eqref{+prop-eq} follow from Corollary~\ref{1var-cor} and \eqref{delta-eq}.
For part (a), note that if $\lambda=(n)$ and $\mu=\emptyset$ then $\tilde\lambda\doublebar\mu = \emptyset\doublebar\emptyset$,
and we have $\mGP_{ \emptyset\doublebar\emptyset} = \mGQ_{ \emptyset\doublebar\emptyset}=1$ and $\varphi(1)=1$.
For part (b), note that if $\varphi$ is positive then $\varphi^+$ is a union of positive specializations, and hence positive,
while if $\varphi^+$ is positive then \eqref{+prop-eq} implies that $\varphi$ is positive.
\end{proof}

Corollaries~\ref{1xx-cor} and \ref{ns-cor} specialize when $\beta=-1$ to the following results.
For $f \in \mGamma$ let $f(1^n)$ denote the variable substitution $f(a_1,a_2,\dots,a_n)$ where $a_1=a_2=\dots=a_n=1$.

\begin{corollary} 
If $\lambda$ and $\mu$ are strict partitions and $n \in \NN$ then
\[
\mGP_{\lambda\doublebar \mu}(1^n)=\mGQ_{\lambda\doublebar \mu}(1^n) = \begin{cases}
1&\text{if }\mu=(\lambda_{n+1},\lambda_{n+2},\lambda_{n+3},\dots) \\
0 &\text{otherwise}.
\end{cases}
\]
\end{corollary}


\begin{corollary}\label{ns-cor2}
If $\lambda$ is a strict partition then
$\mGP_{\lambda }(1^n)=\mGQ_{\lambda}(1^n)=
\begin{cases} 1 & \text{if }n\geq \ell(\lambda) \\
0 &\text{if }n< \ell(\lambda).
\end{cases}$
\end{corollary}

Taking the limiting case of the preceding corollary gives the following:

\begin{corollary}
There exists a specialization $\one:\mGamma\to\RR$
with $\one(\mG_\lambda)=\one(\mGP_\mu)= \one(\mGQ_\mu)=1$ for all partitions $\lambda$ and all strict partitions $\mu$.
\end{corollary}

\begin{proof}
There exists a specialization $\one : \mGamma\to\RR$ with the formula $\one(\mG_\lambda) = \lim_{n\to\infty} \mG_\lambda(1^n)=1$ by \cite[Rem.~6.10]{Yel20}. This specialization also has $\one(\mGP_\mu)= \one(\mGQ_\mu)=1$ by Corollary~\ref{ns-cor2}.
\end{proof}

\subsection{Harmonic functions}\label{harmonic-sect}

In this final section we describe some applications.
Suppose $\cP$ is a directed graph with a unique source vertex $\zero$,
in which
every vertex has finite out-degree, such that
 there is a finite path from $\zero$ to any other vertex $\lambda \in \cP$.
We write $\lambda \to \nu$ if there is an edge in $\cP$ from a vertex $\lambda \in \cP$ to $\nu \in \cP$.
A function $\varphi : \cP \to \RR_{\geq 0}$ is \defn{harmonic}
if 
\[\varphi(\zero)=1\quand \varphi(\lambda) = \sum_{\substack{ \nu \in \cP \\ \lambda \to \nu}} \varphi(\nu).\]
Let $H(\cP)$ be the set of harmonic functions $\cP \to \RR_{\geq 0}$.
This set is \defn{convex} since if $\varphi_1,\varphi_2 \in H(\cP)$ and $c \in [0,1]$ 
then $c \varphi_1 + (1-c) \varphi_2 \in H(\cP)$.
 Let $\partial H(\cP)$ be the set \defn{extreme points} $\varphi \in H(\cP)$ 
 that cannot be expressed as a convex linear combination of harmonic functions with all nonzero coefficients.

\begin{lemma} \label{A-lem}
Suppose there exists  an $\RR$-algebra $A$ with basis $\{ a_\lambda : \lambda \in \cP\}$
such that 
\ben
\item[(P1)] it holds that $a_\zero = 1$ and $a_\lambda a_\mu \in A_+ :=  \RR_{\geq 0} \spanning\{ a_\nu : \nu \in \cP\}$ for all $\lambda,\mu \in \cP$, and 
\item[(P2)] there exists an index $\one \in \cP$ such that $\ds a_\one a_\lambda = \sum_{\substack{\nu \in \cP \\ \lambda\to\nu}} a_\nu$ for each $\lambda \in \cP$.
\een
Fix a map $\varphi : \cP\to \RR$ and let $\rho_\varphi : A \to \RR$ be the linear map with $\rho_\varphi(a_\lambda) = \varphi(\lambda)$.
Then:
\ben
\item[(a)]   $\varphi \in H(\cP)$ if and only if $\rho_\varphi(a_\lambda)\geq0$, $\rho_\varphi(a_\zero)=1$, and $\rho_\varphi(a_\one a_\lambda) =\rho_\varphi(a_\lambda)$
for all $\lambda \in \cP$.

\item[(b)]   $\varphi \in \partial H(\cP)$ if and only if $\rho_\varphi$ is a specialization  
with $\rho_\varphi(a_\lambda)\geq0$ for all $\lambda \in \cP$ and $\rho_\varphi(a_\one)= 1$.

\een
\end{lemma}


Notice that condition (P2) with $\lambda=\zero$ implies that $\zero\to\one$ is the unique edge with source $\zero$.

\begin{proof}
This lemma may be viewed as a special case of \cite[Ex.~4.2]{BO17}.
Alternatively, write $\lambda \vdash n$ if the shortest path from $\zero$  to $\lambda\in\cP$
has $n$ edges. Then  condition (P2) implies that we have $(a_\one)^n - a_\lambda \in A_+$
and one may deduce the lemma by exactly repeating the proof of \cite[Thm.~5.2]{Yel20},
replacing the symbols $\tilde\YY$, $\tilde G_{(1)}$, $\tilde G_{\lambda}$, and $\Gamma_+$ used there by $\cP$, $a_\one$, $a_\lambda$, and $A_+$.
\end{proof}

One may use the preceding lemma to recover several results classifying the extreme  harmonic functions
on directed graphs associated to integer partitions.
Here are two classical examples. Let $\YY$ denote \defn{Young's lattice}, the directed graph whose vertices consist of all integer partitions $\lambda$
and whose edges have the form $\lambda \to \nu$ whenever $\lambda\subseteq \nu$ and $|\nu|-|\lambda|=1$.
A specialization $\rho$ of $\Sym$ is \defn{normalized} if $\rho(h_1)=1$ (recall that $s_{(1)}=h_1$).
 
\begin{corollary}[\cite{VershikKerov}]
A map $\varphi : \YY \to \RR$
belongs to $\partial H(\YY)$ if and only if the linear map $\Sym\to\RR$ sending $s_\lambda \mapsto \varphi(\lambda)$
is a normalized Schur positive specialization.
\end{corollary} 

\begin{proof}
Apply Lemma~\ref{A-lem} (using \eqref{G-pieri} with $\beta=0$) when $\cP= \YY$, $A = \Sym$, and $a_\lambda = s_\lambda$.
\end{proof}

Let $\SY$ denote the shifted variant of Young's lattice, given by the directed graph whose vertices consist of all strict partitions $\lambda$
and whose edges have the form $\lambda \to \nu$ for all strict partitions $\lambda\subseteq \nu$ with $|\nu|-|\lambda|=1$.
A specialization $\varphi$ of $\SymSh$ is  \defn{normalized} if $ \varphi(h_1)=1$ (recall that $P_{(1)}=h_1$).
 
\begin{corollary}
A map $\varphi : \SY \to \RR_{\geq 0}$
belongs to $\partial H(\SY)$ if and only if the linear map $\SymSh\to\RR$ sending $P_\lambda \mapsto \varphi(\lambda)$
is a normalized Schur $P$-positive specialization.
\end{corollary} 

\begin{proof}
Apply Lemma~\ref{A-lem} (using \eqref{GP-pieri} with $\beta=0$) when $\cP= \SY$, $A = \SymSh$, and $a_\lambda = P_\lambda$.
\end{proof}

Yeliussizov \cite[\S5]{Yel20} considered the following \defn{filtered Young graph} $\widetilde\YY$,
whose vertices are the same as $\YY$ but which contains an edge $\lambda \to \nu$
whenever $\lambda\subseteq\nu$ and $\D_\nu\setminus \D_\lambda$ is a nonempty rook strip (that is, having at most
one position in each row and column).
This is part of the \defn{M\"obius deformation} of $\YY$ in the terminology of \cite{PP18}.

A specialization $\varphi$ of $\pGamma$ is  \defn{normalized} if $ \varphi(\pG_1)=1$.
The extreme harmonic functions on $\widetilde\YY$ are given as follows.

\begin{corollary}[\cite{Yel20}]
A map $\varphi : \widetilde\YY \to \RR_{\geq 0}$
belongs to $\partial H(\widetilde\YY)$ if and only if the linear map $\pGamma\to\RR$ sending $\pG_\lambda \mapsto \varphi(\lambda)$
is a normalized $\pG$-positive specialization.
\end{corollary} 

\begin{proof}
In view of \eqref{G-pieri} this follows by 
applying Lemma~\ref{A-lem} with $\cP= \widetilde\YY$, $A = \pGamma$, and $a_\lambda = \pG_\lambda$.
\end{proof}

We introduce a shifted variant: let $\widetilde\SY$ be the graph
with the same vertices as $\SY$ but with edges  $\lambda \to \nu$
whenever $\lambda\subseteq\nu$ are strict partitions such that $\SD_\nu\setminus \SD_\lambda$ is a nonempty rook strip.
A specialization $\varphi$ of $\pGammaP$ is  \defn{normalized} if $ \varphi(\pGP_1)=1$.

\begin{corollary}\label{main-new-cor}
A map $\varphi : \widetilde\SY \to \RR_{\geq 0}$
belongs to $\partial H(\widetilde\SY)$ if and only if the linear map $\pGammaP\to\RR$ sending $\pGP_\lambda \mapsto \varphi(\lambda)$
is a normalized $\pGP$-positive specialization.
\end{corollary} 

\begin{proof}
In view of \eqref{GP-pieri} this follows 
from Lemma~\ref{A-lem} with $\cP= \widetilde\SY$, $A = \pGammaP$, and $a_\lambda = \pGP_\lambda$.
\end{proof}

For convenience, we mention this corollary of Theorem~\ref{main-thm2}. Recall that $\overline{x} = \frac{-x}{1+x}$.

\begin{corollary}
A specialization $\rho : \pGammaP \to \RR$ is normalized and $\pGP$-positive 
if and only if
there are real numbers 
$a=(1\geq a_1\geq a_2 \geq \dots\geq0)$ and $\gamma\geq 0$
with
$
\gamma =  \log 2 -\sum_{n=1}^\infty \log(1+a_n)
$
and
\[ \rho(1+ \pGP_1)^2 + 
\sum_{n=1}^\infty \rho(2\pGP_{n} + 3\pGP_{n+1} +\pGP_{n+2}) z^{n}
 =  4 e^{2\gamma z} \prod_{n=1}^\infty \frac{1-\overline{a_n} z}{1-a_n z}.
\]
In this case $\rho$ coincides with the restriction of $\wphi_a \sqcup \wpi_\gamma$.
\end{corollary}

\end{document}